\documentclass{amsart}

\usepackage[dvips]{graphicx}
\usepackage{amssymb,amsmath}
\usepackage[all]{xy}
\usepackage{enumitem}
\usepackage{makecell} 

\usepackage[dvipdfmx]{hyperref}

\newtheorem{theorem}{Theorem}[section]
\newtheorem{proposition}[theorem]{Proposition}
\newtheorem{lemma}[theorem]{Lemma}

\theoremstyle{definition}
\newtheorem{definition}[theorem]{Definition}
\newtheorem{example}[theorem]{Example}

\theoremstyle{remark}
\newtheorem{remark}[theorem]{Remark}

\numberwithin{equation}{section}

\newcommand{\Cl}{\mathrm{C\ell}}

\title{Variants of Coxeter quandles associated with Pin groups}
\author{Yuichi Kabaya}
\dedicatory{Dedicated to Tomoyoshi Yoshida on the occasion of his 80th birthday}
\address{Faculty of Engineering, Kitami Institute of Technology,
165 Koen-cho Kitami, Hokkaido, JAPAN}
\email{kabaya@mail.kitami-it.ac.jp}
\thanks{}

\begin{document}
\maketitle

\begin{abstract}
We study two families of quandles arising from Coxeter quandles.
One is the quandle defined by Andruskiewitsch-Gra\~{n}a, 
which is the set of roots with binary operation
defined by using the negatives of reflections.
We observe that this is realized as a conjugation quandle in a Pin group.
The other, which we call a rotational $D_n$ quandle, 
is the set of some right angle rotations in the Coxeter group of type $D_n$
with binary operation given by conjugation.
We determine their inner automorphism groups, and observe that they are quite similar. 
\end{abstract}

\section{Introduction}
Quandles are sets with binary operation satisfying some axioms.
Since these axioms correspond to Reidemeister moves of knot diagrams, 
quandles have been intensively studied in low-dimensional topology.

Let $W$ be a Coxeter group.
The \emph{Coxeter quandle} $Q_W$ is a set of all reflections of $W$ 
with the binary operation $*$ is given by conjugation; 
for $x, y \in Q_W (\subset W)$, $x*y = y^{-1} x y$.
For simplicity, we assume that $W$ is finite in this introduction. 
(See later sections for general Coxeter groups.) 
In this case, $W$ acts on the Euclidean space $\mathbb{R}^n$ preserving 
the inner product $(,)$.
For non-zero vector $\beta \in \mathbb{R}^n$, 
a reflection with respect to $\beta$ is defined by 
\[
\alpha \mapsto \alpha * \beta = \alpha - 2 \frac{(\alpha, \beta)}{\{ \beta \|^2} \beta
\quad (\alpha \in \mathbb{R}^n).
\]
For a finite Coxeter group $W$, 
there is an associated finite set $\Phi \subset \mathbb{R}^n$, called a root system, 
satisfying the following properties:
$W$ is generated by the reflections $\{ * \beta \mid \beta \in \Phi \}$, 
$W$ preserves $\Phi$ setwise, and $\alpha = c \alpha$ if and only if $c = \pm 1$.
It is known that $\Phi \to Q_W$, $\beta \mapsto *\beta $ is a two-to-one map
and the $*$ operation on $\Phi$ reduces to the quandle operation of $Q_W$.
On the other hand, $(\Phi, *)$ is not a quandle, but it has a rack structure 
\cite{AG}, \cite{Akita}. $(\Phi, *)$ is called a \emph{Coxeter rack}.

Andruskiewitsch-Gra\~{n}a defined a quandle associated to any rack
\cite{AG}. 
In particular, they observed that the quandle associated 
to a Coxeter rack $(\Phi, *)$ has the binary operation 
\[
\alpha \widetilde{*} \beta = - \alpha * \beta
= - \alpha + 2 \frac{(\alpha, \beta)}{\{ \beta \|^2} \beta
\quad (\alpha, \beta \in \Phi).
\]
We will observe that 
the above definition of $\widetilde{*}$ is also obtained by pulling back the reflection 
$* \beta$  under the double covering $\mathrm{Pin}(n) \to \mathrm{O}(n)$.
So we call the quandle $(\Phi, \widetilde{*})$ the \emph{double covering} of $Q_W$
and denote it by $DQ_W$. The quandle $DQ_W$ is not faithful, 
and in many cases, its inner automorphism group $\mathrm{Inn}(DQ_{W})$ 
is isomorphic to $W$ 
(see Proposition \ref{prop:inner_auto_of_DQ_W} for the precise statement). 

We introduce another quandle $Q^r_{D_n}$ related to the Coxeter group of type $D_n$.
The quandle structure of $Q^r_{D_n}$ is easily described in terms of a Clifford algebra.
Let $\{ e_1, \cdots, e_n \}$ be the standard basis of $\mathbb{R}^n$.
The Clifford algebra $\Cl (\mathbb{R}^n)$ is the algebra generated by $e_1, \cdots , e_n$ 
with relations $e_i e_j + e_j e_i = - 2(e_i, e_j)$. 
Let $Q^r_{D_n} = \{ e_i e_j \mid i \neq j \} = \{ \pm e_i e_j\mid i < j\} 
\subset \Cl(\mathbb{R}^n)$.
Define
\[
(e_i e_j) * (e_k e_l) 
= \begin{cases} 
e_i e_j & ( \textrm{if $| \{ i, j \} \cap \{ k , l \} | \equiv 0 \mod 2$} \, ) \\
e_i e_j e_k e_l & ( \textrm{if $| \{ i, j \} \cap \{ k , l \} | \equiv 1 \mod 2$} \, )
\end{cases}.
\]
It is not difficult to directly show that $(Q^r_{D_n}, *)$ satisfies the quandle axioms, 
but we will show that, under the following identification, 
\[
Q^r_{D_n} \ni e_i e_j \longleftrightarrow (1 + e_ie_j)/\sqrt{2} \in \mathrm{Spin}(n),
\]
$Q^r_{D_n}$ is a quandle by conjugation in $\mathrm{Spin}(n)$.
Here we remark that 
$(1 + e_ie_j)/\sqrt{2} = \cos \dfrac{\pi}{4} + e_ie_j \, \sin \dfrac{\pi}{4}$ 
gives a $\pi/2$ rotation of $\mathbb{R}^n$ via $\mathrm{Spin}(n) \to \mathrm{SO}(n)$.
Since $Q^r_{D_n} = \{ \pm e_i e_j\mid i < j\}$ 
is one-to-one corresponds to the set of positive root of type $D_n$, 
$\{ e_i \pm e_j \mid i < j \}$ by $e_i e_j \longleftrightarrow e_i \pm e_j$, 
we call $Q^r_{D_n}$ a \emph{rotational $D_n$ quandle}.
We will show that the inner automorphism group $\mathrm{Inn}(Q^r_{D_n})$, 
in many cases, is isomorphic to the Coxeter group $W$
(see Theorem \ref{thm:inner_automorphism_group_of_Q^r_D_n} for the precise statement).

This paper is organized as follows.
In \S \ref{sec:quandles}, we will recall the definition of quandles and some terminologies. 
In \S \ref{sec:coxeter_quandles}, we will review Coxeter groups and Coxeter quandles.
In particular, we determine the inner automorphism group of a Coxeter quandle.
We will introduce the double covering $DQ_W$ in \S \ref{sec:double_coverings} and 
the rotational $D_n$ quandle $Q^r_{D_n}$ in \S\ref{sec:rotational_D_n}.
In the final section \S \ref{sec:vendramin's_classification}, 
we will give a list of $Q_W$, $DQ_W$ and $Q^r_{D_n}$ 
which appear in Vendramin's classification of small connected quandles.

\medskip\medskip
\noindent
\textbf{Acknowledgments.}
This work was supported by JSPS KAKENHI Grant Number JP23K03082.

\section{Quandles}
\label{sec:quandles}
\begin{definition}
A \emph{quandle} is a set $X$ with a binary operation $* : X \times X \to X$ satisfying 
\begin{enumerate}[label={(Q\arabic*)}]
\item
\label{axiom_1}
$x * x = x$ $(x \in X)$,
\item
\label{axiom_2}
$\forall y \in X$, $* y : X \to X$, \,  $x \mapsto x*y$ is a bijection,
\item
\label{axiom_3}
$(x*y)*z = (x*z) * (y*z)$ \, $(x, y, z \in X)$.
\end{enumerate}
If it satisfies the axioms \ref{axiom_2} and \ref{axiom_3} but may not \ref{axiom_1}, 
$(X,*)$  is called a \emph{rack}.
\end{definition}

Let $G$ be a group. 
For $g, h \in G$, we define $*$ by $g*h = h^{-1} g h$.
It is easy to check that $*$ satisfies \ref{axiom_1} and \ref{axiom_3}.
Thus, for a subset $S \subset G$ closed under conjugation 
(i.e. $g, h \in S$ implies $h^{-1} g h, h g h^{-1} \in S$), 
$(S, *)$ is a quandle.


For quandles $X, Y$, a map $f: X \to Y$ is called a homomorphism 
if it satisfies $f(x * y) = f(x) * f(y)$.
A bijective homomorphism is called an isomorphism. 
Remark that the inverse map of an isomorphism is also an isomorphism.
Thus the set of all isomorphisms $X \to X$ forms a group with multiplication 
given by composition of isomorphisms. 
This groups is called the \emph{automorphism group} of $X$, 
and denoted by $\mathrm{Aut}(X)$.

For any $y \in X$, $*y : X \to X$ $x \mapsto x*y$ is 
an isomorphism by \ref{axiom_2} and  \ref{axiom_3}.
The subgroup of $\mathrm{Aut}(X)$ generated by $\{ *y \mid y \in X\}$ 
is called an \emph{inner automorphism group}, and denoted by $\mathrm{Inn} (X)$.
If $X$ is finite, $\mathrm{Aut}(X)$ and $\mathrm{Inn}(X)$ are explicitly realized as 
subgroups of the permutation group of degree $|X|$.
When $X$ is a rack, 
$\mathrm{Aut}(X)$ and $\mathrm{Inn}(X)$ are defined in the same way.

If the map $X \to \mathrm{Inn} (X)$, $y \mapsto *y$ is injective, 
$X$ is called \emph{faithful}.
If $\mathrm{Inn}(X)$ acts on $X$ transitively, $X$ is called \emph{connected} 
(or indecomposable).

Let $(R,*)$ be a rack.
For any $y \in R$, we denote the inverse of $*y$ by $*^{-1}y$. 
Let $x, y, z \in R$.
Since $( (x *^{-1} z ) * (y *^{-1} z) ) * z  = x * y$, 
we have $(x *^{-1} z ) * (y *^{-1} z) = (x * y) *^{-1} z$.
This means that $*^{-1} z$ is an automorphism of $(R, *)$.
Since 
\[
\begin{split}
\Bigl( \Bigl( (x *^{-1} y ) *^{-1} z \Bigr) * (y *^{-1} z) \Bigr) * z 
= (x *^{-1} y) * y = x, \\
\Bigl( \Bigl( (x *^{-1} z ) *^{-1} (y *^{-1} z) \Bigr) * (y *^{-1} z) \Bigr) * z  
= (x *^{-1} z) * z = x, 
\end{split}
\]
we have $(x *^{-1} y ) *^{-1} z = (x *^{-1} z ) *^{-1} (y *^{-1} z)$. 
Thus $(R, *^{-1})$ is a rack.
Since $ \Bigl( ( x *^{-1} y) * z \Bigr) * (y*z) =  ( ( x *^{-1} y) * y ) * z  = x *z$,  
we have $( x *^{-1} y) * z = (x*z) *^{-1} (y * z)$. 
This means that $*z$ is an automorphism of $(R, *^{-1})$.
Since $(x * (y*^{-1}y) ) * y  = (x*y ) * y$, we have
\begin{equation}
\label{eq:fiber}
x * (y*^{-1}y)  = x*y. 
\end{equation}

For a rack $(R, *)$, 
define $*_{\iota} : R \times R \to R$ by
\begin{equation}
\label{eq:associated_quandle}
x *_{\iota} y := ( x *^{-1} x) * y \quad (x,y \in R).
\end{equation}
Then $(R, *_{\iota})$ clearly satisfies \ref{axiom_1} and \ref{axiom_2}.
(The inverse of $*_{\iota}y$ is given by $x *_{\iota}^{-1} y = (x*x)*^{-1} y$.)
For $x, y, z \in R$, using (\ref{eq:fiber}), we have
\[
\begin{split}
(x *_{\iota} z) *_{\iota} (y *_{\iota} z)
&= \Bigl( \Bigl( ( x *^{-1} x) * z \Bigr) *^{-1} \Bigl( ( x *^{-1} x) * z \Bigr) \Bigr) * ((y*^{-1}y)*z) \\
&= \Bigl( \Bigl( ( x *^{-1} x) *^{-1} ( x *^{-1} x) \Bigr) * z \Bigr) * ((y*^{-1}y)*z) \\
&= \Bigl( \Bigl( ( x *^{-1} x) *^{-1} ( x *^{-1} x) \Bigr) * (y*^{-1}y) \Bigr)  * z\\
&= \Bigl( \Bigl( ( x *^{-1} x) *^{-1} ( x *^{-1} x) \Bigr) * y \Bigr)  * z\\
&= \Bigl( (( x *^{-1} x) * y) *^{-1} (( x *^{-1} x)*y) \Bigr) * z \\
&= (x *_{\iota} y) *_{\iota} z.
\end{split}
\]
Thus $(R, *_{\iota})$ also satisfies \ref{axiom_3}. 
The quandle $(R, *_{\iota})$ is defined by 
Andruskiewitsch-Gra\~{n}a \cite[\S 1.1.1]{AG}, 
called the \emph{quandle associated to} a rack $(R.*)$.

\section{Coxeter quandles}
\label{sec:coxeter_quandles}
\subsection{Coxeter groups and quandles}
\label{subsec:coxeter_groups_and_quandles}
A \emph{Coxeter group} $W$ is a group generated by $S = \{s_1, \cdots , s_n\}$ subject to 
the relations
\[
(s_1)^2 = 1, \, \cdots \, , (s_n)^2 = 1, \quad
(s_is_j)^{m_{i j}} = 1 \quad (i \neq j), 
\]
where $m_{i j} = m_{j i} \in \{2, 3, 4, \cdots , \infty \}$. 
Here $(s_is_j)^{\infty} = 1 $ means no relation for the pair $s_i$, $s_j$.
The pair $(W, S)$ is called a \emph{Coxeter system}.
If we set $m_{ii} = 1$ $(i = 1, \cdots , n)$, 
we have $W = \langle s_1, \cdots , s_n \mid (s_i s_j)^{m_{i j}} = 1 \rangle$.
The number $n = |S|$ is called the \emph{rank} of $(W, S)$.

The \emph{Coxeter graph} of a Coxeter system $(W, S)$ 
is a simplicial graph with edge labeling,  
defined by the following rules.
\begin{itemize}
\item The vertices one-to-one correspond to $s_1, \cdots , s_n$.
\item $s_i$, $s_j$ $(i \neq j)$ are joined by an edge if $m_{ij} \geq 3$. 
\item Each edge is labeled by $m_{ij}$. 
(If the edge label is $m_{ij} = 3$, it is usually omitted.)
\end{itemize}
Conversely, a simplicial graph $\Gamma$ with edge labeling 
by $\{ 3, 4, \cdots , \infty\}$ uniquely determine 
a Coxeter system $(W_{\Gamma}, S_{\Gamma})$.
If a Coxeter diagram $\Gamma$ is decomposed into 
connected components as $\Gamma = \Gamma_1 \sqcup \cdots \sqcup \Gamma_r$, 
then $W_{\Gamma}$ is isomorphic to the direct product 
$W_{\Gamma_1} \times \cdots \times W_{\Gamma_r}$.
If the Coxeter graph is connected, 
the corresponding Coxeter group is called \emph{irreducible}.

Let $\{ \alpha_1, \cdots, \alpha_n \}$ be a basis of $V = \mathbb{R}^n$.
Define a symmetric bilinear form on $V$ by requiring 
\begin{equation}
\label{eq:symmetric_bilinear_form}
B(\alpha_i, \alpha_j) = - \cos \frac{\pi}{m_{ij}}.
\end{equation}
This is interpreted as $B(\alpha_i, \alpha_j) = -1$ when $m_{ij} = \infty$.
$\alpha_i$ is a unit vector with respect to $B$ since $m_{ii} = 1$.
For each $s_i$, we define a linear map $\sigma_{s_i} : V \to V$ by 
\[
\sigma_{s_i} (x) = x - 2 B(\alpha_{i}, x) \alpha_{i} \quad (x \in V).
\] 
It is easy to check that $\sigma_{s_i}$ preserves the form $B$, 
and $\sigma_{s_i}$ is an involution.
There is a unique homomorphism 
$\sigma : W \to \mathrm{GL}(V)$ which sends $s_i$ to $\sigma_{s_i}$
\cite[Proposition 5.3]{Humphreys}.
Moreover, this representation is faithful \cite[Corollary 5.4]{Humphreys}.
So we may regard $W$ as a subgroup of $\mathrm{GL}(V)$ 
and write $\sigma(w) (x)$ as $w(x)$ for $x \in V$, $w \in W$. 

We define the \emph{root system} of $(W, S)$ by 
\[
\Phi = \{ w (\alpha_i) \mid w \in W, \, i = 1, \cdots , n \} \subset V.
\]
Clearly, $\Phi$ is closed under the action of $W$. 
Since $W$ preserves $B$, any element of $\Phi$ is a unit vector.
Since $\{ \alpha_1, \cdots, \alpha_n \}$ is a basis of $V$, 
we can write any $\alpha \in \Phi$ uniquely of the form  
$\alpha  = \sum_{i=1}^n c_{\alpha,i} \alpha_i$ ($c_{\alpha, i} \in \mathbb{R}$).
If $c_{\alpha, i} \geq 0$ (resp. $c_{\alpha, i} \leq 0$) for all $i$, 
$\alpha$ is called \emph{positive} (resp. \emph{negative}).
Let $\Phi^+$ (resp. $\Phi^-$) be the set of all positive (resp. negative) roots of $\Phi$.
It is known that $\Phi^- = - \Phi^+$ and $\Phi$ is a disjoint union of 
$\Phi^{+}$ and $\Phi^{-}$, thus $\Phi = \Phi^+ \sqcup (-\Phi^+)$ \cite[\S 5.4]{Humphreys}.

For $\alpha = w(\alpha_i)  \in \Phi$, we associate 
$s_{\alpha} =  w \sigma_{s_i} w^{-1} \in W \subset \mathrm{GL}(V)$.
Since 
\begin{equation}
\label{eq:reflection}
s_{\alpha} (x) = w (\sigma_{s_i} (w^{-1} x))
= w ( w^{-1} x  - 2 B(\alpha_i, w^{-1} x) \alpha_i )
= x - 2 B(\alpha, x) \alpha,
\end{equation}
$s_{\alpha}$ depend only on $\alpha$, 
not on the choice of $w$ and $\alpha_{i}$.
Since $s_{w (\alpha)}(x) = x - 2 B(w (\alpha), x) w (\alpha)
= w( w^{-1} (x) - 2 B(\alpha, w^{-1} (x)) \alpha) 
= w s_{\alpha} w^{-1} (x)$ for $x \in V$, we have
\[
s_{w \alpha} = w s_{\alpha} w^{-1}, \textrm{ in particular, } 
s_{s_\beta(\alpha)} = s_{\beta} s_{\alpha} s_{\beta}^{-1}.
\]
We have $s_{\alpha}(\alpha) = - \alpha $, 
and $s_{\alpha}(x) = x$ for $x \in V$ perpendicular to $\alpha$ with respect to $B$.
Thus $s_{\alpha}$ is a reflection with respect to the hyperplane perpendicular to $\alpha$. 
Let $Q_W$ be the set of such reflections; 
\[
Q_W = \{ s_\alpha \mid \alpha \in \Phi \} 
= \bigcup_{w \in W} \{ w \sigma_{s_i} w^{-1} \mid  i = 1, \cdots, n \}.
\]
Since $Q_W \subset W$ is closed under conjugation, $Q_W$ has a quandle structure by 
$s_{\alpha} * s_{\beta} = s_{\beta}^{-1} s_{\alpha} s_{\beta}$.
$Q_W$ is called a \emph{Coxeter quandle}.

For $\alpha, \beta \in \Phi$, let
\begin{equation}
\label{eq:star_operation}
\alpha * \beta = s_{\beta} (\alpha) = 
\alpha - 2 B(\alpha, \beta) \beta.
\end{equation}
For $\alpha, \beta, \gamma \in \mathbb{R}^n$, we have
\begin{equation}
\label{eq:distributive}
\begin{split}
(\alpha * \gamma) * (\beta * \gamma)
&= \alpha * \gamma 
- 2 B (\alpha*\gamma, \beta*\gamma)\beta * \gamma
= \alpha * \gamma - 2 B (\alpha, \beta) \beta * \gamma  \\
&= (\alpha - 2 B(\alpha, \beta) \beta ) * \gamma
= (\alpha * \beta ) * \gamma, 
\end{split}
\end{equation}
since $\sigma_{\gamma}$ is a linear map preserving $B$.
Thus $(\Phi, *)$ satisfies \ref{axiom_3}, and since $s_\alpha \in \mathrm{GL}(V)$, 
it also satisfies \ref{axiom_2}. 
Thus $(\Phi, *)$ is a rack, but $(\Phi, *)$ is not a quandle 
since $\alpha * \alpha = s_{\alpha}(\alpha) = - \alpha$.
The rack $(\Phi, *)$ is called a \emph{Coxeter rack} \cite{AG}, \cite{Akita}.

The map $p : \Phi \to Q_W$ defined by $p(\alpha) = s_{\alpha}$ is clearly surjective. 
We have $s_{-\alpha} = s_{\alpha}$, moreover, $p$ is a two-to-one map. 
In fact, if $s_{\alpha} = s_{\beta}$,
$-\beta = s_{\beta} (\beta) = s_{\alpha} (\beta) = \beta - 2 B(\alpha, \beta) \alpha$
means that $\beta = B(\alpha, \beta) \alpha$.
This means that $\beta = \pm \alpha$ since $\alpha$, $\beta$ are unit vectors.
Thus $p$ induces a bijection $\Phi/\{ \alpha \sim -\alpha \} \xrightarrow{\cong} Q_W$.
Since
\[
p(\alpha * \beta) = p(s_{\beta} (\alpha)) = s_{s_{\beta} (\alpha)} 
= s_{\beta} s_\alpha s_{\beta}^{-1} = s_{\beta}^{-1} s_\alpha s_{\beta} = p(\alpha) * p(\beta), 
\]
this map is a quandle isomorphism \cite{Akita}.
So we may regard $Q_W$ as $\Phi/\{ \alpha \sim -\alpha \}$.

For $\alpha, \beta \in \Phi$ and $w \in W$, 
\[
w (\alpha * \beta) 
= w ( \alpha - 2 B(\beta, \alpha) \beta)
= w (\alpha) - 2 B(w(\beta), w(\alpha)) w (\beta)
= w(\alpha) * w(\beta),
\]
means that $W$ acts on $\Phi$.
Since $\Phi$ contains the basis $\{ \alpha_1, \cdots , \alpha_n \}$, 
this action is faithful.
Thus we can regard $W$ as a subgroup of $\mathrm{Aut}(\Phi)$.
Since $W$ is generated by $s_{\alpha_i} = \sigma_{s_i}$, 
we have $\mathrm{Inn} (\Phi) = W$.
In summary, we have the following:
\begin{lemma}
\label{lem:inner_automorphism_group_of_coxeter_rack}
Let $\Phi$ be the Coxeter rack of a Coxeter group $W$.
Then $\mathrm{Inn} (\Phi) \cong W$.
\end{lemma}

\begin{lemma}
\label{lem:inner_automorphism_group_of_coxeter_quandle}
Let $Q_W$ be a Coxeter quandle of an irreducible Coxeter group $W$.
Then $\mathrm{Inn} (Q_W) \cong W / \{ \pm 1 \}$ if $W$ contains $-1$, 
or $\mathrm{Inn} (Q_W) \cong W$ otherwise.
\end{lemma}
\begin{proof}
The action of $W$ on $\Phi$ induces an action of $W$
on $\Phi/\{ \alpha \sim -\alpha \} \cong Q_W$, and we have a surjective homomorphism 
$W \cong  \mathrm{Inn} (\Phi) \twoheadrightarrow \mathrm{Inn} (Q_W)$.

Let $w \in W$ be an element in the kernel 
$W \cong \mathrm{Inn} (\Phi) \twoheadrightarrow \mathrm{Inn} (Q_W)$.
Since $w$ is in the kernel, we have $w \alpha_i = \pm \alpha_i$  in $\Phi$ 
for $i = 1, \cdots , n$. 
We will show that these have all of the same sign. 
This means that $w = \pm 1$, thus concludes the proof.

First, we assume that $w \alpha_1 = \alpha_1$.
Let $\alpha_k$ be one of $\alpha_2, \cdots , \alpha_n$ 
which is connected to $\alpha_1$ by an edge in the Coxeter graph. 
By (\ref{eq:symmetric_bilinear_form}) and $m_{1k} \neq 2$, 
we have $B(\alpha_1, \alpha_k) < 0$. 
By (\ref{eq:reflection}), 
$s_{\alpha_k} (\alpha_1) = \alpha_1 + c \, \alpha_k \in \Phi$ 
where $c$ is a positive real number.
We have $w (\alpha_1 + c \, \alpha_k) = w \alpha_1 + c \, w \alpha_k 
= \alpha_1 \pm c \, \alpha_k \in \Phi$.
But $\alpha_1 - c \, \alpha_k$ is never in $\Phi^+ \sqcup \Phi^- = \Phi$, 
we conclude that $w \alpha_k = \alpha_k$.
Since every vertex in the Coxeter graph is connected to the vertex corresponding 
to $\alpha_1$, we conclude that $w \alpha_i = \alpha_i$ for all $i = 1, \cdots , n$.

When $w \alpha_1 = -\alpha_1$, a similar argument works.
\end{proof}

From Lemma \ref{lem:inner_automorphism_group_of_coxeter_quandle}, 
we obtain the following.
\begin{proposition}
\label{prop:inner_automorphism_group_of_coxeter_quandle}, 
Let $Q_W$ be a Coxeter quandle of a Coxeter group $W$.
Let $\Gamma$ be the Coxeter graph of $W$, and 
$\Gamma = \Gamma_1 \sqcup \cdots \sqcup \Gamma_r$ be the decomposition into 
connected components.
Then $\mathrm{Inn} (Q_W) \cong W'_{\Gamma_1} \times \cdots \times W'_{\Gamma_r}$
where $W'_{\Gamma_i} \cong W_{\Gamma_i} / \{ \pm 1 \}$ if $W_{\Gamma_i}$ contains $-1$, 
or $W'_{\Gamma_i} = W_{\Gamma_i}$ otherwise.
\end{proposition}
\begin{proof}
As in the proof of Lemma \ref{lem:inner_automorphism_group_of_coxeter_quandle}, 
let $w \in W$ be an element in the kernel 
$W \cong \mathrm{Inn} (\Phi) \twoheadrightarrow \mathrm{Inn} (Q_W)$.
Let $\{ \alpha_{i, 1}, \cdots , \alpha_{i, n_i} \} \subset 
\{ \alpha_1, \cdots , \alpha_n \}$ be the set of vectors corresponding to 
the vertices of $\Gamma_i$.
As in the proof of Lemma \ref{lem:inner_automorphism_group_of_coxeter_quandle}, 
we can show that $w \alpha_{i, j} = \pm \alpha_{i, j }$ $(j = 1, \cdots , n_i)$ 
have all of the same sign. 
After rearranging the order of basis vectors, we have
\[
\mathrm{Ker} (\mathrm{Inn} (\Phi) \twoheadrightarrow \mathrm{Inn} (Q_W)) 
= W \cap \left\{ 
\begin{pmatrix} 
\pm I_{n_1} & & O \\ 
& \ddots &  \\ 
O &  & \pm I_{n_r} 
\end{pmatrix} 
\right\}.
\]
This concludes the proof.
\end{proof}

We remark that $\mathrm{Inn} (Q_W)$ is realized as a subgroup of 
$\mathrm{PGL}(n_1, \mathbb{R}) \times \cdots \times \mathrm{PGL}(n_r, \mathbb{R} )$
in Proposition \ref{prop:inner_automorphism_group_of_coxeter_quandle}.

\subsection{Finite Coxeter groups}
\label{subsec:finite_coxeter_groups}
It is known that $(W, S)$ is finite 
if and only if the symmetric form $B$ is positive definite
\cite[Theorem 6.4]{Humphreys}.

All irreducible finite Coxeter groups are classified into the following types:
\begin{itemize}
\item Infinite families: $A_n$, $BC_n$, $D_n$, $I_2(m)$
\item Exceptional ones: $E_6$, $E_7$, $E_8$, $F_4$, $H_3$, $H_4$
\end{itemize}
For example, we write the Coxeter group of type $A_n$ by $W_{A_n}$, 
its root system by $\Phi_{A_n}$, and its Coxeter quandle by $Q_{A_n}$.

Accompanied with Proposition \ref{prop:inner_automorphism_group_of_coxeter_quandle}, 
the following fact is fundamental \cite[Corollary 3.19]{Humphreys}.
\begin{proposition}
\label{prop:contains_-1}
Let $W$ be an irreducible finite Coxeter group.
$-1 \notin W$ if and only if $W$ is one of types $A_n$  ($n \geq 2$), 
$D_n$ ($n$ : odd), $E_6$, or $I_2(m)$ ($m$ : odd).
\end{proposition}

For a finite Coxeter group, we use the following definition of the root system.
A finite subset $\mathrm{\Phi} \subset \mathbb{R}^n \setminus \{ 0 \}$ is called 
an \emph{abstract root system} if it satisfies 
\begin{itemize}
\item
$\alpha, \beta \in \mathrm{\Phi} \Longrightarrow  
s_{\beta}(\alpha) \in \mathrm{\Phi}$,
\item
$\beta = c \alpha  \,\,\, (\alpha, \beta \in \mathrm{\Phi}, \, c \in \mathbb{R}) 
\Longrightarrow c = \pm 1$,
\item
the linear span of $\Phi$ is $\mathbb{R}^n$.
\end{itemize}
Here, $s_{\beta} (x) = x - \frac{2 (x, \beta)}{(\beta, \beta)} \beta$ 
is a reflection with respect to the hyperplane 
perpendicular to $\beta$.
An element of $\mathrm{\Phi}$ is called a \emph{root}.
Fix a total ordering $\leq$ on $\mathbb{R}^n$ satisfying 
(i) $x \leq y \Longrightarrow x + z \leq y + z$ $(x, y, z \in \mathbb{R}^n)$, and
(ii) $x \leq y \Longrightarrow c x \leq c y$ and $ -c x \geq -c y$
$(x, y \in \mathbb{R}^n, \, c > 0)$.
If we let $\Phi^{+} = \{ \alpha \in \Phi \mid \alpha > 0 \}$, 
and $\Phi^{-} = \{ \alpha \in \Phi \mid \alpha < 0 \}$, 
then it is known that $\Phi$ is a disjoint union of $\Phi^{+}$ and $\Phi^{-}$.
There exists a unique subset $\{ \alpha_1, \cdots , \alpha_n \} \subset \Phi^{+}$ 
such that any $\alpha \in \Phi^+$ is uniquely written as a linear combination of 
$\{ \alpha_1, \cdots , \alpha_n \}$ with non-negative coefficients 
\cite[Theorem 1.3]{Humphreys}.
$\alpha_1, \cdots , \alpha_n$ are called \emph{simple roots}.
Let $W$ be the group generated by $S = \{ s_{\alpha_1}, \cdots , s_{\alpha_n} \}$.
It is known that $(W, S)$ is a finite Coxeter group.
Conversely, any finite Coxeter group is written in this form by some abstract root system.
After normalizing the lengths of vectors of $\Phi$, 
this coincides with the set of roots in \S \ref{subsec:coxeter_groups_and_quandles}.

\begin {example}
\label{ex:dihedral_rack_and_quandle}
For an integer $m \geq 2$, 
$\Phi_{I_2(m)} = 
\{ (\cos \dfrac{k \pi}{m}, \, \sin \dfrac{k \pi}{m}) \}_{k =0, 1, \cdots, 2m -1} 
\subset \mathbb{R}^2$ forms a root system of $I_2(m)$.
We have
$(\cos \dfrac{i \pi}{m}, \, \sin \dfrac{i \pi}{m})
* (\cos \dfrac{j \pi}{m}, \, \sin \dfrac{j \pi}{m})
= (\cos \dfrac{i \pi}{m}, \, \sin \dfrac{i \pi}{m})
 - 2 \cos \dfrac{(i-j) \pi}{m} (\cos \dfrac{j \pi}{m}, \, \sin \dfrac{j \pi}{m})
= ( -\cos \dfrac{(i - 2j) \pi}{m}, \, \sin \dfrac{(i -2j)\pi}{m})
= ( \cos \dfrac{(2j - i + m) \pi}{m}, \, \sin \dfrac{(2j - i + m)\pi}{m})$.
If we identify $\mathbb{Z} / 2m \mathbb{Z}$ with $\Phi_{I_2(m)}$ by 
$k \longleftrightarrow (\cos \dfrac{k \pi}{m}, \, \sin \dfrac{k \pi}{m})$, 
the rack structure on $\mathbb{Z} / 2m \mathbb{Z}$ is given by
 $ i * j \equiv 2 j - i + m\mod 2 m$.
Thus the Coxeter quandle $Q_{I_2(m)}$ is regarded as $\mathbb{Z}/ m \mathbb{Z}$ 
with $i*j \equiv 2 j -i \mod m$.




\end{example}

\section{Double coverings of Coxeter Quandles}
\label{sec:double_coverings}
\subsection{Double coverings of Coxeter Quandles}
Let $\Phi$ be a root system of $W$ and 
$*$ be the binary operation defined by (\ref{eq:star_operation}).
Since $(\Phi, *)$ is a rack, we have the associated quandle (\ref{eq:associated_quandle}) 
defined by Andruskiewitsch-Gra\~{n}a. 
Since $* \alpha$ ($\alpha \in \Phi$) is an involution, 
we have $*^{-1} \alpha = * \alpha$. 
By $\alpha * \alpha = - \alpha$, we have
$\alpha *_{\iota} \beta = (\alpha *^{-1} \alpha) * \beta = (\alpha * \alpha) * \beta =
(-\alpha) * \beta = - (\alpha * \beta)$ for $\alpha, \beta \in \Phi$.
This explicit form of $*_{\iota}$ was already obtained in \cite[\S 1.3.5]{AG}.

Since the quandle $(\Phi, *_{\iota})$ can be seen as a double covering of $Q_W$, 
we denote this quandle by $(DQ_W, \widetilde{*})$.
In other words, we let $DQ_W = \Phi$ and 
\begin{equation}
\label{eq:tilde_*}
\alpha \widetilde{*} \beta :=  -(\alpha * \beta) 
= -s_{\beta} (\alpha) = -\alpha + 2 B(\alpha, \beta) \beta.
\end{equation}
We collect some basic identities:
\[
\begin{split}
(-\alpha) * \beta = - (\alpha * \beta), \quad \alpha * (-\beta) = \alpha * \beta, \quad
\alpha * \alpha = - \alpha, \\
(-\alpha) \widetilde{*} \beta = - (\alpha \widetilde{*} \beta), 
\quad \alpha \widetilde{*} (-\beta) = \alpha \widetilde{*} \beta, 
\quad \alpha \widetilde{*} \alpha = \alpha.
\end{split}
\]
So we do not need to distinguish $(-\alpha) * \beta$ and $-(\alpha * \beta)$
(also $(-\alpha) \widetilde{*} \beta$ and $- (\alpha \widetilde{*} \beta)$).
Although we already observed that $(DQ_W, \widetilde{*})$ is a quandle, 
it is easy to check \ref{axiom_3} directly form the definition (\ref{eq:tilde_*})
using these identities:
\[
\begin{split}
(\alpha \widetilde{*} \beta) \widetilde{*} \gamma 
&= - (- \alpha * \beta) * \gamma
= (\alpha * \beta) * \gamma  
= (\alpha * \gamma) * (\beta * \gamma) \\
&= -(-\alpha * \gamma) * (- \beta * \gamma ) 
= (\alpha \widetilde{*} \gamma) \widetilde{*} (\beta \widetilde{*} \gamma).
\end{split}
\]
We call the quandle $(DQ_W, \widetilde{*})$ the \emph{double covering of $Q_W$}. 
Since $\alpha \widetilde{*} (-\beta) = \alpha \widetilde{*} (\beta)$, $DQ_W$ 
is not faithful.
The natural quotient map
\[
DQ_W = \Phi \longrightarrow \Phi /  \{ \alpha \sim - \alpha \} = Q_W
\]
is clearly a two-to-one surjective quandle homomorphism.
We will see that when $W$ is finite,  
$DQ_W$ can be realized as the inverse image of $Q_W \subset \mathrm{O}(n)$ 
under the natural surjective homomorphism 
$\mathrm{Pin}(n) \twoheadrightarrow  \mathrm{O}(n)$.

\begin{remark}
The notion of quandle coverings was defined by Eisermann \cite{Eisermann}.
A \emph{quandle covering} is a surjective homomorphism $p : \widetilde{Q} \to Q$ 
of quandles $\widetilde{Q}$, $Q$ if $p(\widetilde{x}) = p(\widetilde{y})$ implies 
$\widetilde{a} * \widetilde{x} = \widetilde{a} * \widetilde{y}$ for all 
$\widetilde{a}, \widetilde{x}, \widetilde{y} \in \widetilde{Q}$.
The map $DQ_{W} \to Q_W$ is a quandle covering.
\end{remark}

Let $\Gamma$ be the Coxeter diagram of $W$.
Let $\Gamma_{\mathrm{odd}}$ be a graph 
whose vertices coincides with $\Gamma$ and two vertices $s_i$ and $s_j$  ($i \neq j$)
are connected by an edge if $m_{i j}$ is odd.
It is shown in \cite[Proposition 2.1]{Akita} that
$Q_W$ is connected if and only if $\Gamma_{\mathrm{odd}}$ is connected. 
For $DQ_W$, we have the following.

\begin{proposition}
Let $W$ be a Coxeter group of rank $n$ and $\Gamma$ be its Coxeter diagram.
$DQ_W$ is connected if \textup{(i)} $\Gamma_{\mathrm{odd}}$ is connected,  and 
\textup{(ii)} for any $\alpha \in DQ_W = \Phi$, there exists $\beta \in \Phi$ 
perpendicular to $\alpha$. 
\end{proposition}
\begin{proof}
Let $\alpha, \alpha' \in DQ_W$.
Since $Q_W = DQ_W /  \{ \alpha \sim - \alpha \}$ is connected, 
there exists a sequence $\beta_1, \cdots , \beta_k \in DQ_W$ such that 
$\alpha' 
= \pm ( ( \cdots  ( (\alpha \widetilde{*} \beta_1) \widetilde{*} \beta_2 ) 
\widetilde{*} \cdots ) \widetilde{*} \beta_k ) $ in $DQ_{W}$.
If the sign is $+$, $\alpha$ and $\alpha'$ is clearly in the same orbit under 
the action of $\mathrm{Inn} (DQ_W)$.
Thus we assume that $\alpha'
= - ( ( \cdots  ( (\alpha \widetilde{*} \beta_1) \widetilde{*} \beta_2 ) 
\widetilde{*} \cdots ) \widetilde{*} \beta_k )$.
Take $\gamma \in \Phi$ perpendicular to $\alpha'$.
Since $\alpha' * \gamma = \alpha'$, we have $\alpha' \, \widetilde{*} \gamma = -\alpha'$.
Thus
$\alpha' = - \alpha' \, \widetilde{*} \gamma
= ( ( \cdots  ( (\alpha \widetilde{*} \beta_1) \widetilde{*} \beta_2 ) 
\widetilde{*} \cdots ) \widetilde{*} \beta_k ) \widetilde{*} \gamma$ 
implies $\alpha$ and $\alpha'$ is in the same orbit under 
the action of $\mathrm{Inn} (DQ_W)$.
\end{proof}

We remark that the condition (ii) is satisfied for any irreducible finite Coxeter group 
except dihedral groups (including $W_{A_2}$).
On the other hand, 
from Example \ref{ex:dihedral_rack_and_quandle},
we have $DQ_{I_2 (m)} \cong Q_{I_2 (2m)}$, and thus $DQ_{I_2 (m)}$ is not connected.

\begin{proposition}
\label{prop:inner_auto_of_DQ_W}
Let $W$ be a Coxeter group of rank $n$.
$\mathrm{Inn} (DQ_{W})$ is isomorphic to $W$ if $-1 \notin W$ or $n$ is even.
When $W$ is finite, 
$\mathrm{Inn} (DQ_{W}) \cong W$ if and only if $-1 \notin W$ or $n$ is even.
\end{proposition}
\begin{proof}
Let $\Phi$ be a root system of $W$, 
and $s_{\alpha}$ be the reflection with respect to $\alpha \in \Phi$.
As in the proof of Lemma \ref{lem:inner_automorphism_group_of_coxeter_rack}, 
we observe that $\mathrm{Inn} (DQ_W)$ is the group generated by 
$\{ -s_{\alpha} \mid \alpha \in \Phi \}$ in $\mathrm{Aut} (\Phi)$, 
i.e. $\mathrm{Inn} (DQ_W) = \langle \{ -s_{\alpha} \mid \alpha \in \Phi \} \rangle$.

We remark that, since $s_{\alpha}$ is an orientation reversing linear isomorphism, 
$\det ( - s_{\alpha}) > 0$ when $n$ is odd, 
and $\det ( - s_{\alpha}) < 0$ when $n$ is even.
Thus 
$\mathrm{Inn} (DQ_W) 
< \mathrm{GL}^+(n, \mathbb{R}) = \{ A \in \mathrm{GL}(n, \mathbb{R}) \mid \det A > 0\}$
when $n$ is odd.

First, we assume that $-1 \notin W$.
It is known that there exists a surjective homomorphism
$\varepsilon : W \to \{ \pm 1\}$ sending each $s_\alpha$ to $-1$ 
\cite[Proposition 5.1]{Humphreys}.
We define a homomorphism 
$\tau : W = \langle \{ s_{\alpha} \mid \alpha \in \Phi \} \rangle 
\to \langle \{ -s_{\alpha} \mid \alpha \in \Phi \} \rangle$
by $\tau(w)  = \varepsilon(w) w$ for $w \in W$.
$\tau$ is surjective and 
satisfies $\tau(s_\alpha)  = - s_{\alpha}$ for $\alpha \in \Phi$.
We show that $\tau$ is also injective.
Let $s_{\beta_1} s_{\beta_2} \cdots s_{\beta_k} \in \mathrm{Ker} \, \tau$ 
where $\beta_1, \cdots , \beta_k \in \Phi$.
We have $1 = \tau(s_{\beta_1} s_{\beta_2} \cdots s_{\beta_k}) 
= (-1)^k s_{\beta_1} s_{\beta_2} \cdots s_{\beta_k}$, 
thus $s_{\beta_1} s_{\beta_2} \cdots s_{\beta_k} = (-1)^k$.
Since $-1 \notin W$, $k$ must be even and
$s_{\beta_1} s_{\beta_2} \cdots s_{\beta_k} = 1$.
We conclude that $\tau$ is an isomorphism.

Next, we assume that $-1 \in W$.
We have
\[
W = \langle \{ s_{\alpha} \mid \alpha \in \Phi \} \rangle
= \langle \{ s_{\alpha} \mid \alpha \in \Phi \}, -1  \rangle
= \langle \{ -s_{\alpha} \mid \alpha \in \Phi \}, -1  \rangle
\geq \langle \{ -s_{\alpha} \mid \alpha \in \Phi \} \rangle
= \mathrm{Inn} (DQ_W).
\]
Here $-1 \in W$ means that there exist $\beta_1 , \cdots , \beta_k \in \Phi$
such that $-1 = s_{\beta_1} \cdots s_{\beta_k}$.
If $n$ is even, 
since $\det (-1) > 0$ and $\det s_{\beta_i} < 0$, $k$ must be even.
Thus we have $-1 = s_{\beta_1} \cdots s_{\beta_k} = (-s_{\beta_1}) \cdots (-s_{\beta_k})
\in \mathrm{Inn} (DQ_W)$ and conclude that $W = \mathrm{Inn} (DQ_W)$.
If $n$ is odd, since $-1 \notin \mathrm{GL}^+(n, \mathbb{R})$, 
we have $W \gneq \mathrm{Inn} (DQ_W)$. 
If $W$ is finite, $\mathrm{Inn} (DQ_W)$ is not isomorphic to $W$. 
\end{proof}

By Proposition \ref{prop:contains_-1}, 
if $W$ is an irreducible finite Coxeter group other than $BC_n$ ($n$ : odd), $E_7$ or $H_3$, 
we have $\mathrm{Inn} (DQ_{W}) \cong W$.
By a GAP calculation \cite{Gap}, 
we observe that $\mathrm{Inn} (DQ_{E_7}) \cong \mathrm{O}(7,2)$, 
$W_{E_7} \cong C_2 \times \mathrm{O}(7,2)$, 
$\mathrm{Inn} (DQ_{H_3}) \cong A_5$, 
and $W_{H_3} \cong C_2 \times A_5$, 
where $C_2$ is the cyclic group of order $2$, 
$A_5$ is the alternating group of order $5$, 
and $\mathrm{O}(7,2)$ is the orthogonal group of the 7 dimensional vector space 
over the finite field $\mathbb{F}_2$.
Apart from irreducible ones, for example, we have 
$W_{D_4 \times E_7}  \gneq \mathrm{Inn} (DQ_{D_4 \times E_7})$.

\subsection{Clifford algebras}
Let $V$ be an $n$-dimensional real vector space with an inner product 
$(,) : V \times V \to \mathbb{R}$.
Let $\displaystyle T = \bigoplus_{k \geq 0} V^{\otimes k}$ be the tensor algebra, 
and $I$ the two-sided ideal of $T$ generated by all elements of the form
$v \otimes w + w \otimes v =-2 (v, w)$ where $v, w \in V$.
The quotient $\Cl(V) := T/I$ is called the \emph{Clifford algebra}.
If we fix an orthonormal basis $\{ e_1, \cdots , e_n \}$ of $V$,
then $e_{i_1} e_{i_2} \cdots e_{i_k}$ ($i_1 < i_2 < \cdots < i_k$) form
a vector space basis of $\Cl(V)$.
For example, we have the following relations:
\[
e_i e_j = - e_j e_i \,\, (i \neq j), \quad e_i^2 = -1.
\]
Let $\Cl ^{0}(V)$ (resp. $\Cl ^{1}(V)$) be the subspace of $\Cl(V)$ 
spanned by products of even (resp. odd) number of $e_1, \cdots , e_n$. 
Clearly, $\Cl (V) = \Cl ^{0}(V) \oplus \Cl ^{1}(V)$ and 
$\Cl ^{0}(V)$ is a subalgebra of $\Cl(V)$.

There exists a unique automorphism $\alpha : \Cl (V) \to \Cl (V)$ which extends
the map $\alpha(v) = - v$ ($v \in V$) on $V$ \cite[p. 9]{LawsonMichelsohn}.
We have $\alpha|_{\Cl ^{0}(V)} = \mathrm{id}$ and $\alpha|_{\Cl ^{1}(V)} = -\mathrm{id}$.

\subsection{Pin and Spin}
For $n \geq 2$, $\mathrm{Pin}(n)$ is the multiplicative subgroup of $\Cl(V)$ 
generated by elements of $V$ of norm $1$.
The \emph{spin group} $\mathrm{Spin}(n)$ is defined by 
$\mathrm{Spin} (n) = \mathrm{Pin}(n) \cap \Cl^{0}(V)$.
Define a right action of $\mathrm{Spin}(n)$ on $V$ by
\begin{equation}
\label{eq:spin_action}
v \cdot s = s^{-1} v s
\quad (v \in V, \, s \in \mathrm{Spin}(n) ).
\end{equation}
This is an orientation preserving linear isometry of $V$, 
thus gives a homomorphism $\rho : \mathrm{Spin}(n) \to \mathrm{SO}(n)$.
(Because the inner automorphism group of a quandle naturally acts on the quandle
from the right, we prefer right action.)
It is known that, $\rho$ is a surjective two-to-one map.

For example, for $s = (\cos \theta) 1 + (\sin \theta) e_i e_j$ 
($i \neq j$, $\theta \in \mathbb{R}$), 
$v \mapsto s^{-1} v s$ is the $-2 \theta$-rotation about the origin
on the plane $\mathbb{R}e_i \oplus \mathbb{R} e_j$, 
and the identity on its orthogonal complement.

For $v, w \in V$ ($w \neq 0$), we define $v \cdot w \in V$ by
\[
v \, \cdot \, w = w^{-1} v \, \alpha(w) = w^{-1} v (-w) 
= \frac{1}{\| w \|^2} w v w
= \frac{1}{\| w \|^2} (-v w - 2 (v,w)) w
= v - \frac{2 (v,w)}{\| w \|^2} w.
\]
(We remark that $w^{-1} = -\frac{w}{\|w\|^2}$ for non-zero $w \in V$.)
The map $v \mapsto v \cdot w$ is the reflection with respect to 
the hyperplane perpendicular to $w$.
By definition, any element $s$ of $\mathrm{Pin}(n)$ can be written as 
a product $s = v_1 v_2 \cdots v_k$ by some unit vectors $v_1, \cdots , v_k \in V$.
Define the action of $\mathrm{Pin}(n)$ on $V$ by
\[
v \cdot s = ( \cdots ( (v \cdot v_1) \cdot v_2)  \cdots ) \cdot v_k 
= s^{-1} v \alpha(s)
\quad (v \in V, \,\, s = v_1 \cdots v_k \in \mathrm{Pin}(n)).
\]
The map $v \mapsto v \cdot s$ is an element of $\mathrm{O}(n)$
since it is a composition of reflections. 
This gives a surjective homomorphism $\rho : \mathrm{Pin}(n) \to \mathrm{O}(n)$
since $\mathrm{O}(n)$ is generated by reflections.
The restriction of $\rho$ to $\mathrm{Spin}(n)$ coincides with (\ref{eq:spin_action}).
It is known that $\rho$ is two-to-one map \cite[p. 19]{LawsonMichelsohn}.
For non-zero element $w \in V$, the inverse image of 
$ \, \cdot \, w \in \mathrm{O}(n)$ is 
$\{ \pm \frac{w}{\| w \|} \} \subset \mathrm{Pin}(n)$.

Let $W$ be a finite Coxeter group of rank $n$ and $\Phi$ be its root system.
The inverse image of 
$Q_{W} = \{ s_{\alpha} = *\alpha \mid \alpha \in \Phi \}\subset \mathrm{O}(n)$
under $\mathrm{Pin}(n) \to \mathrm{O}(n)$ 
is $\{ \frac{\alpha}{\| \alpha \|} \mid  \alpha \in \Phi \}$.
We can freely change a scalar factor of $\Phi$, we regard 
$\widetilde{Q}_W = \{ \alpha  \mid  \alpha \in \Phi \}$
as the inverse image of $Q_W$.
$\widetilde{Q}_W$ has a quandle structure by conjugation
\[
\alpha * \beta = \beta^{-1} \alpha \beta
= - \frac{1}{\| \beta \|^2} \beta \alpha \beta
= - \frac{1}{\| \beta \|^2} (-\alpha \beta - 2 (\alpha, \beta))\beta
= -\alpha + \frac{2 (\alpha, \beta))}{\| \beta \|^2} \beta.
\]
This is the negative of the reflection with respect to $\beta$.
Thus the quandle $\widetilde{Q}_W$ by conjugations is isomorphic 
to the double cover $DQ_W$.

\begin{remark}
The Pin group $\mathrm{Pin}(n)$ is also denoted by $\mathrm{Pin}_-(n)$.
If we define the Clifford multiplication by a positive definite form 
\[
v w + w v = 2(v ,w),
\]
the corresponding Pin group is denoted by $\mathrm{Pin}_+(n)$.
We have also a two-to-one surjective homomorphism $\mathrm{Pin}_+(n) \to \mathrm{O}(n)$.
The inverse image of $Q_W \subset \mathrm{O}(n)$ in $\mathrm{Pin}_+(n)$ forms 
a quandle by conjugations. In this case, we have 
\[
\alpha * \beta = \beta^{-1} \alpha \beta
= \frac{1}{\| \beta \|^2} \beta \alpha \beta
= \frac{1}{\| \beta \|^2} (-\alpha \beta + 2 (\alpha, \beta))\beta
= -\alpha + \frac{2 (\alpha, \beta))}{\| \beta \|^2} \beta.
\]
Thus this is also isomorphic to $DQ_W$.
But we remark that the inverse images of $W < \mathrm{O}(n)$ in 
$\mathrm{Pin}_{+}(n)$  and $\mathrm{Pin}_{-}(n)$ may differ.
\end{remark}

\begin{remark}
Even if $W$ is not finite, we can define the Clifford algebra associated with
the symmetric bilinear form $B$ defined in \S \ref{subsec:coxeter_groups_and_quandles}. 
We can also define $\mathrm{Pin}(n, B)$ and $\mathrm{Spin}(n, B)$ 
and $\mathrm{O}(n, B)$, and regard $DQ_W$ as an inverse image of 
$W < \mathrm{O}(n, B)$ in $\mathrm{Pin}(n, B)$.
\end{remark}

\section{Rotational $D_n$ quandles}
\label{sec:rotational_D_n}
\subsection{Root system of  type $D_n$}
\label{subsec:root_system_of_type_D_n}
Let $n$ be an integer with $n \geq 2$. 
Let $e_1, \cdots , e_n$ be the standard basis of $\mathbb{R}^n$.
Then
\[
\Phi_{D_n} = \{ \pm (e_i \pm e_{j}) \mid i < j \}
\]
forms a root system in the sense of \S \ref{subsec:finite_coxeter_groups}. 
This is the root system of type $D_n$.
If we use the ascending lexicographic order, 
the set of positive roots is $\{ e_i \pm e_j  \mid i < j \}$, 
and the set of simple roots is 
$\{ e_1 - e_2, \, e_2 - e_3, \, \cdots , \, e_{n-1} - e_n, \, e_{n-1} + e_n \}$.
There are $2 n(n-1)$ roots, $n (n-1)$ positive roots and $n$ simple roots.
Usually, $D_2$ is recognized as $A_1 \times A_1$ and $D_3$ as $A_3$.

Similarly, 
\[
\Phi_{B_n} = \{ \pm ( e_i \pm e_j ) \mid i < j \} 
\sqcup \{ \pm e_i \mid i = 1, \cdots , n \} 
\]
is a root system, called a root system of type $B_n$.
The former set of long (length $\sqrt{2}$) roots coincides with $\Phi_{D_n}$,
and the latter set of short (length $1$) roots is a root system of $(A_1)^{n}$.
Let $\Omega = \{ \pm e_1, \cdots , \pm e_n\}$ be the set of short roots of $\Phi_{B_n}$.
Since any reflection $s_{\alpha}$ $(\alpha \in \Phi_{B_n})$ preserves the set $\Omega$, 
$W_{B_n}$ acts on $\Omega$.
Moreover, this action is faithful since $\Omega$ contains a basis of $\mathbb{R}^n$ 
and $W_{B_n} < \mathrm{O}(n)$. 
The natural projection 
$\Omega = \{ \pm e_1 , \cdots , \pm e_n \} \to 
\overline{\Omega} = \{ e_i , \cdots , e_n \} :
\pm e_i \mapsto e_i$ induces an action of $W_{B_n}$ on $\overline{\Omega}$.
Since $s_{e_i + e_j}$ acts on $\overline{\Omega}$ by transposition, 
$W_{B_n}$ acts on $\overline{\Omega}$ as the full permutation group. 
Thus we have a surjective homomorphism 
$W_{B_n} \twoheadrightarrow \mathrm{Aut}(\overline{\Omega}) \cong S_n$ 
where $S_n$ is the symmetric group of degree $n$.
Let $K$ be the kernel of $W_{B_n} \twoheadrightarrow S_n$.
Since any element of $K$ only changes signs of $\{ \pm e_1 , \cdots , \pm e_n\}$, 
$K$ can be regarded as a subgroup of $(\mathbb{Z} / 2 \mathbb{Z})^n$.
Since $s_{e_i}$ changes only the sign of the $i$-th slot, we conclude that 
$K = (\mathbb{Z} / 2 \mathbb{Z})^n$.
Thus we have a short exact sequence 
\[
0 \to (\mathbb{Z} / 2 \mathbb{Z})^n \to W_{B_n} \to S_n \to 1.
\]
For $\sigma \in S_n$, $\widetilde{\sigma} (\pm e_i) = \pm \sigma (e_i)$ gives 
a splitting of this short exact sequence.
Thus $W_{B_n}$ is isomorphic to a semi-direct product 
$(\mathbb{Z} / 2 \mathbb{Z})^n \rtimes S_n$.
We can realize $W_{B_n}$ in $\mathrm{O}(n)$ explicitly as follows.
For $\sigma \in S_n$ and $\varepsilon_i = \pm 1$ $(i = 1, \cdots n)$, 
we define an $n \times n$ matrix $M(\sigma, \{ \varepsilon_i \})$ by 
\begin{equation}
\label{eq:definition_of_M_sigma_epsilon}
M(\sigma, \{ \varepsilon_i \} ) = (\varepsilon_i \delta_{i, \sigma(j)})_{i, j}
\end{equation}
where $\delta_{i, j}$ is the Kronecker delta. 
Then we have
\begin{equation}
W_{B_n} = \left\{ M(\sigma , \{ \varepsilon_i \})
\mid \sigma \in S_n , \,\, \varepsilon_i = \pm 1 \, (i = 1, \cdots n) \right\}.
\end{equation}
For example, 
\[
W_{B_2} = \left\{  
\pm \begin{pmatrix} 1 & 0 \\ 0 & 1 \end{pmatrix}, 
\pm \begin{pmatrix} 1 & 0 \\ 0 & -1 \end{pmatrix}, 
\pm \begin{pmatrix} 0 & 1 \\ 1 & 0 \end{pmatrix}, 
\pm \begin{pmatrix} 0 & -1 \\ 1 & 0 \end{pmatrix}
\right\}  < \mathrm{O}(2).
\]

Next, we determine the group structure of $W_{D_n}$.
Since $\Phi_{D_n} \subset \Phi_{B_n}$, $W_{D_n}$ is a subgroup of $W_{B_n}$.
We also have a surjective homomorphism 
$W_{D_n} \twoheadrightarrow \mathrm{Aut}(\overline{\Omega}) \cong S_n$.
Again, the kernel of this homomorphism is a subgroup of $(\mathbb{Z} / 2 \mathbb{Z})^n$, 
but $s_{e_i \pm e_j}$ changes only even number of signs.
Thus $W_{D_n}$ is isomorphic to a semi-direct product 
$(\mathbb{Z} / 2 \mathbb{Z})^{n-1} \rtimes S_n$.
Explicitly, we have
\begin{equation}
\label{eq:D_n_Weyl_group}
W_{D_n} = \left\{ M(\sigma , \{ \varepsilon_i \})
\mid \sigma \in S_n , \,\, \varepsilon_i = \pm 1 \, (i = 1, \cdots n)
, \,\, \varepsilon_1 \varepsilon_2 \cdots \varepsilon_n = 1 \right\}.
\end{equation}
If we define a homomorphism $\mathrm{sign}$ by
\[
\begin{array}{cccc}
\mathrm{sign} : & W_{B_n} & \to & \{ \pm 1 \}  \\
& \rotatebox[origin=c]{90}{$\in$} & & \rotatebox[origin=c]{90}{$\in$} \\
& M(\sigma, \{ \varepsilon_i \}) & \mapsto 
& \varepsilon_1 \varepsilon_2 \cdots \varepsilon_n
\end{array},
\]
then $W_{D_n}$ is the kernel of this homomorphism.

By Proposition \ref{prop:contains_-1}, we know that $-1 \in W_{D_n}$ 
if and only if $n$ is even, 
but this is easily checked by the description (\ref{eq:D_n_Weyl_group}). 
As a consequence, we determine the structure of the inner automorphism group 
of a Coxeter quandle $Q_{D_n}$.

\begin{proposition}
Let $Q_{D_n}$ be a finite Coxeter quandle of type $D_n$. Then 
\[
\mathrm{Inn} (Q_{D_n}) \cong 
\begin{cases}
(\mathbb{Z}/2 \mathbb{Z})^{n-1} \rtimes S_n & \textrm{(if $n$ is odd),} \\
(\mathbb{Z}/2 \mathbb{Z})^{n-2} \rtimes S_n & \textrm{(if $n$ is even and $n \neq 2$),} \\
\{ 1 \} & \textrm{(if $n=2$).}
\end{cases}
\]
\end{proposition}

\subsection{Rotational $D_n$ quandles}
Let 
\[
Q^r_{D_n} = \{ e_i e_j \mid i \neq j \} = \{ \pm e_i e_j \mid i < j \} \subset \Cl (\mathbb{R}^n).
\]
We remark that elements of $Q^r_{D_n}$ are symmetric or skew-symmetric.
\begin{lemma}
\label{lem:symmetric_or_skew_symmetric}
For $i \neq j$ and $k \neq l$, 
\begin{equation}
\begin{split}
(e_i e_j) (e_k e_l) &= 
\begin{cases}
(e_k e_l) (e_i e_j) & \textrm{(if $i$, $j$, $k$, $l$ are distinct)}\\
(e_k e_l) (e_i e_j) = -1 & \textrm{(if $i = k$ and $j = l$)}\\
(e_k e_l) (e_i e_j) = 1 & \textrm{(if $i = l$ and $j = k$)}\\
-(e_k e_l) (e_i e_j) & \textrm{(otherwise)}
\end{cases} \\
&= 
\begin{cases}
(e_k e_l) (e_i e_j) & (\textrm{if $| \{ i, j \} \cap \{ k , l \} | \equiv 0 \mod 2$} \, ) \\
-(e_k e_l) (e_i e_j) & (\textrm{if $| \{ i, j \} \cap \{ k , l \} | \equiv 1 \mod 2$} \, )
\end{cases}.
\end{split}
\end{equation}
\end{lemma}
Assume that $i \neq j$ and $k \neq l$.
From Lemma \ref{lem:symmetric_or_skew_symmetric},  we have
\begin{equation}
\label{eq:action_on_the_wedge_product}
\begin{split}
& \frac{1}{\sqrt{2}} (1 - e_k e_l) \cdot e_i e_j \cdot \frac{1}{\sqrt{2}} (1 + e_k e_l) \\
&= \frac{1}{2} (e_i e_j - e_k e_l e_i e_j + e_i e_j e_k e_l - e_k e_l e_i e_j e_k e_l) \\
&= \begin{cases} 
e_i e_j & ( \textrm{if $| \{ i, j \} \cap \{ k , l \} | \equiv 0 \mod 2$} \, ) \\
e_i e_j e_k e_l & ( \textrm{if $| \{ i, j \} \cap \{ k , l \} | \equiv 1 \mod 2$} \, )
\end{cases}.
\end{split}
\end{equation}
Since $(1 - e_i e_j) (1 + e_i e_j) = 2$, we also have
\begin{equation}
\label{eq:binary_operation_by_conjugation}
\begin{split}
& \frac{1}{\sqrt{2}} (1 - e_k e_l) \cdot \frac{1}{\sqrt{2}} (1 + e_i e_j) \cdot \frac{1}{\sqrt{2}} (1 + e_k e_l) \\
&= \begin{cases} 
\dfrac{1}{\sqrt{2}} (1 + e_i e_j) 
& ( \textrm{if $| \{ i, j \} \cap \{ k , l \} | \equiv 0 \mod 2$} \, ) \\
\dfrac{1}{\sqrt{2}} (1 + e_i e_j e_k e_l) 
& ( \textrm{if $| \{ i, j \} \cap \{ k , l \} | \equiv 1 \mod 2$} \, ) 
\end{cases}.
\end{split}
\end{equation}
Define
\begin{equation}
\label{eq:definition_of_binary_operation}
\begin{split}
(e_i e_j) * (e_k e_l) 
&= \begin{cases} 
e_i e_j & ( \textrm{if $| \{ i, j \} \cap \{ k , l \} | \equiv 0 \mod 2$} \, ) \\
e_i e_j e_k e_l & ( \textrm{if $| \{ i, j \} \cap \{ k , l \} | \equiv 1 \mod 2$} \, )
\end{cases} \\
&= \begin{cases} 
e_i e_j & ( \textrm{if $(e_i e_j) (e_k e_l) = (e_k e_l) (e_i e_j)$} ) \\
e_i e_j e_k e_l & ( \textrm{if $(e_i e_j) (e_k e_l) = - (e_k e_l) (e_i e_j)$} )
\end{cases}.
\end{split}
\end{equation}
If we put $x = e_i e_j$ and $y = e_k e_l$, this is also written as
\begin{equation}
\label{eq:definition_of_binary_operation_2}
x * y = \frac{x - y x y}{2} + \frac{x y - y x}{2}
= \begin{cases}
x & ( \textrm{if $x y  = y x$} ) \\
x y & ( \textrm{if $x y = - y x$} )
\end{cases}.
\end{equation}
Using (\ref{eq:definition_of_binary_operation}), 
(\ref{eq:binary_operation_by_conjugation}) is written as
\begin{equation}
\frac{1}{\sqrt{2}} (1 - e_k e_l) \cdot \frac{1}{\sqrt{2}} (1 + e_i e_j) \cdot \frac{1}{\sqrt{2}} (1 + e_k e_l)
= \frac{1}{\sqrt{2}} (1 + (e_i e_j)*(e_k e_l) ).
\end{equation}
Explicitly, when $i < j$ and $k < l$, we have
\begin{equation}
(e_i e_j) * (e_k e_l) = 
\left\{
\begin{array}{rll} 
e_k e_j & (k < i = l < j)
& \begin{minipage}{80pt}\includegraphics[width=80pt]{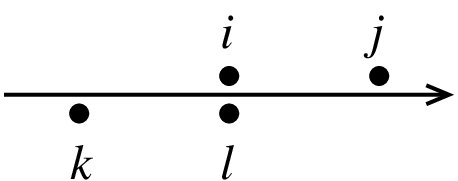}\end{minipage}\\
-e_l e_j & (i = k < l < j) 
& \begin{minipage}{80pt}\includegraphics[width=80pt]{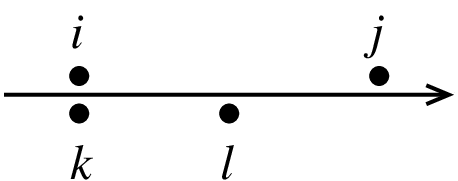}\end{minipage}\\
e_j e_l & (i = k < j < l)
& \begin{minipage}{80pt}\includegraphics[width=80pt]{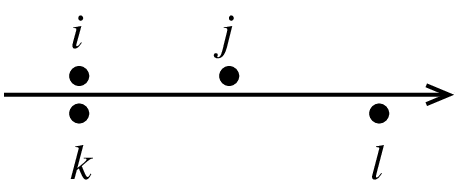}\end{minipage}\\
-e_k e_i & (k < i < j = l)
& \begin{minipage}{80pt}\includegraphics[width=80pt]{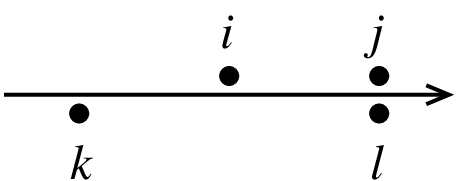}\end{minipage}\\
e_i e_k & (i < k < j = l)
& \begin{minipage}{80pt}\includegraphics[width=80pt]{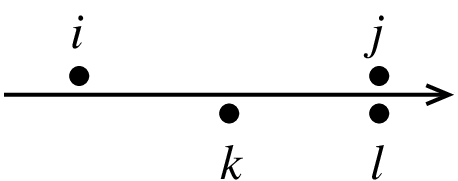}\end{minipage}\\
-e_i e_l & (i < j = k < l)
& \begin{minipage}{80pt}\includegraphics[width=80pt]{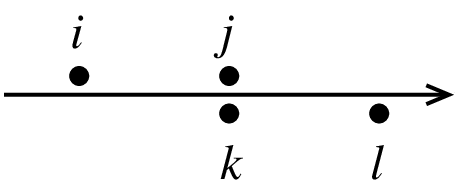}\end{minipage}\\
e_i e_j & \textrm{(otherwise)}
\end{array}
\right.
\end{equation}
In particular, $\{ (1 + e_i e_j)/\sqrt{2} \mid i \neq j \} \subset \mathrm{Spin}(n)$ 
is closed under conjugation. 
Thus we have the following.
\begin{proposition}
Let 
\[
Q'^r_{D_n} = \{ (1 + e_i e_j)/\sqrt{2} \mid i \neq j \}
= \{ (1 \pm e_i e_j)/\sqrt{2} \mid i < j \} \subset \mathrm{Spin}(n).
\]
$Q'^r_{D_n}$ is closed under conjugation. Thus $Q'^r_{D_n}$ is a quandle by conjugation.
\end{proposition} 
By the one-to-one correspondence
\[
Q^r_{D_n} \ni e_i e_j \longleftrightarrow (1 + e_ie_j)/\sqrt{2} \in Q'^r_{D_n},
\]
$Q^r_{D_n} = \{ e_i e_j \mid i \neq j \} = \{ \pm e_i e_j \mid i < j \}$ 
has a quandle structure by (\ref{eq:definition_of_binary_operation}), 
or equivalently (\ref{eq:definition_of_binary_operation_2}).

For a Coxeter quandle $Q_W$, $*\alpha \in \mathrm{Inn}(Q_W)$ $(\alpha \in Q_W)$ 
is an involution. 
On the other hand, for $e_ie_j, \, x \in Q^r_{D_n}$, 
\[
(x * (e_i e_j))* (e_i e_j)
= \begin{cases} 
x & (\textrm{if $x e_ie_j = e_ie_j x$}) \\
x e_i e_j e_i e_j = -x & (\textrm{if $x e_ie_j = - e_ie_j x$})
\end{cases}
\]
means that $*(e_i e_j)$ is not an involution.
Thus $Q^r_{D_n} \cong Q'^r_{D_n}$ can not be a Coxeter quandle.

\subsection{Inner automorphism group of $Q^r_{D_n}$}
By a direct calculation, we have
\begin{equation}
\label{eq:action_of_Q_on_V}
\begin{split}
&(1 - e_i e_j)/\sqrt{2} \cdot e_i \cdot (1 + e_i e_j)/\sqrt{2} = -e_j, \\
&(1 - e_i e_j)/\sqrt{2} \cdot e_j \cdot (1 + e_i e_j)/\sqrt{2} = e_i, \\
&(1 - e_i e_j)/\sqrt{2} \cdot e_k \cdot (1 + e_i e_j)/\sqrt{2} = e_k 
\quad \textrm{($k \notin \{ i, j \}$)},
\end{split}
\end{equation}
for $i \neq j$.
This can also be observed by the fact that the right action of 
$(1 + e_i e_j)/\sqrt{2} \in Q'^r_{D_n}$ is the $-\pi/2$-rotation about the origin 
on $\mathbb{R} e_i \oplus \mathbb{R} e_j$
and the identity on its orthogonal complement.

Let $G$ be the group generated by $Q'^r_{D_n} = \{ (1 + e_i e_j)/\sqrt{2} \mid i \neq j \}$ 
in $\mathrm{Spin}(n)$.
By (\ref{eq:action_of_Q_on_V}), 
we see that $G$ preserves the set $\Omega = \{ \pm e_1, \cdots , \pm e_n\} \subset V$.
Since $\mathrm{SO}(n)$ acts faithfully on $V$ and $\Omega$ contains a basis of $V$, 
$\rho(G)$ acts on $\Omega$ faithfully. 
Thus we can regard $\rho(G) \leq \mathrm{Aut}(\Omega)$.

The natural projection 
$\Omega = \{ \pm e_1 , \cdots , \pm e_n \} \to 
\overline{\Omega} = \{ e_i , \cdots , e_n \} :
\pm e_i \mapsto e_i$ induces an action of $\rho(G)$ on $\overline{\Omega}$.
By (\ref{eq:action_of_Q_on_V}), $(1 + e_i e_j)/\sqrt{2}$ transpose $e_i$ and $e_j$ 
in $\overline{\Omega}$, thus $\rho(G)$ acts on $\overline{\Omega}$ 
as the full permutation group. Thus we have
\[
\begin{array}{ccc}
\mathrm{Aut} (\Omega) & \twoheadrightarrow & \mathrm{Aut} (\overline{\Omega}) \\
\rotatebox[origin=c]{-90}{$\geq$} & & \rotatebox[origin=c]{-90}{$\cong$} \\
\rho(G) & \twoheadrightarrow & S_n
\end{array}
\]
where $S_n$ is the symmetric group of degree $n$.
Let $K$ be the kernel of $\rho(G) \twoheadrightarrow S_n$.
Since any element of $K$ only changes signs of $\{ \pm e_1 , \cdots , \pm e_n\}$, 
$K$ can be regarded as a subgroup of $(\mathbb{Z} / 2 \mathbb{Z})^n$.
We see that $K \leq (\mathbb{Z} / 2 \mathbb{Z})^{n-1}$ 
since $\rho(G) \leq \mathrm{SO}(n)$ implies that there are even number of sign changes.
Moreover, $((1 + e_i e_j)/\sqrt{2} )^2 = e_i e_j$ only changes signs of 
$\pm e_i$ and $\pm e_j$, we conclude that $K = (\mathbb{Z} / 2 \mathbb{Z})^{n-1}$.
Thus we have a short exact sequence 
\begin{equation}
\label{eq:short_exact_sequence_for_rho(G)}
0 \to (\mathbb{Z} / 2 \mathbb{Z})^{n-1} \to \rho(G) \to S_n \to 1.
\end{equation}
In particular, the order of $\rho(G)$ is $2^{n-1} n!$.

As a subset of $\mathrm{SO}(n)$, 
any element of $\rho(G)$ can be written as 
$M(\sigma, \{ \varepsilon_i \} ) = (\varepsilon_i \delta_{i, \sigma(j)})_{i, j}$
(see (\ref{eq:definition_of_M_sigma_epsilon}))
where $\sigma \in S_n$ and $\varepsilon_i = \pm 1$ ($i = 1, \cdots , n$) 
satisfying $\det (\sigma , \{ \varepsilon_i \}) = 1$.
Comparing the cardinalities of the following sets, we obtain
\begin{equation}
\label{eq:rho_G}
\begin{split}
\rho(G) &= 
\left\{ M(\sigma , \{ \varepsilon_i \})
\mid \sigma \in S_n , \,\, \varepsilon_i = \pm 1 \, (i = 1, \cdots n) , \,\, 
\det (\sigma , \{ \varepsilon_i \}) = 1 \right\} \\
&= W_{B_n} \cap\mathrm{SO}(n).
\end{split}
\end{equation}
Combining the results of \S \ref{subsec:root_system_of_type_D_n}, 
we have the following inclusions:
\[
\begin{array}{ccccc}
 \mathrm{O}(n) & > & W_{B_n} & \geq & W_{D_n} \\
\rotatebox[origin=c]{-90}{$>$} & & \rotatebox[origin=c]{-90}{$\geq$} & & \\
\mathrm{SO}(n) & > & \rho(G)  \\
\end{array}
\]
We also have the following short exact sequences:
\begin{equation}
\label{eq:W_and_rho_G_as_kernels}
\begin{array}{ccccccccc}
1 & \to & W_{D_n} & \to & W_{B_n} & \xrightarrow{\mathrm{sgn}} & \{ \pm 1 \} & \to & 1 ,  \\
1 & \to & \rho(G) & \to & W_{B_n} & \xrightarrow{\mathrm{det}} & \{ \pm 1 \} & \to & 1 .  \\
\end{array}
\end{equation}

The right action $V \curvearrowleft \mathrm{Spin}(n)$ induces an action 
$V \wedge V \curvearrowleft \mathrm{Spin}(n)$.
This action reduces to the action of $\mathrm{PSO}(n)$, where 
\[
\mathrm{PSO}(n) = 
\begin{cases}
\mathrm{SO}(n)/ \{ \pm I \} & \textrm{( $n$ : even)} \\
\mathrm{SO}(n) & \textrm{( $n$ : odd)}.
\end{cases}
\]
When $n \geq 3$, the adjoint group of $\mathrm{Spin}(n)$ is $\mathrm{PSO}(n)$, 
thus the action $V \wedge V \curvearrowleft \mathrm{PSO}(n)$ is faithful. 
We identify $V \wedge V$ with $\mathrm{span}_{\mathbb{R}} \{ e_i e_j \mid i < j \}$.
For $i \neq j$ and $k \neq l$, by (\ref{eq:action_on_the_wedge_product}) and 
(\ref{eq:definition_of_binary_operation}),
we have
\[
\frac{1}{\sqrt{2}} (1 - e_k e_l) \cdot e_i e_j \cdot \frac{1}{\sqrt{2}} (1 + e_k e_l) 
= (e_i e_j) * (e_k e_l).
\]
Thus the right action of $\mathrm{Inn}(Q^r_{D_n})$ on $Q^r_{D_n}$ 
can be identified with the action of $\mathrm{Inn}(Q'^r_{D_n})$
on $\{ e_i e_j \mid i \neq j \}  \subset V \wedge V$.
Thus  $\mathrm{Inn}(Q^r_{D_n})$ is realized as a subgroup of $\mathrm{PSO}(n)$.

When $n = 2 k +1$, we have the following diagram:
\[
\xymatrix@C=1pt@R=1pt{
W_{D_{2k+1}} \ar@{>}[rrrrrr]^{1:1} & & & & & & \mathrm{Inn}(Q_{D_{2k+1}}) \\
\rotatebox[origin=c]{-90}{$\lneq$} & & & & & & \rotatebox[origin=c]{-90}{$\leq$} \\
W_{B_{2k + 1}} \ar@/^1.5pc/[rrrrrr]^{2:1} & < & \mathrm{O}(2k+1) & \xrightarrow{2:1} & \mathrm{PO}(2k+1) 
& > &  \overline{W_{B_{2k+1}}} \\
\rotatebox[origin=c]{-90}{$\gneq$}  & & \rotatebox[origin=c]{-90}{$>$} & & \rotatebox[origin=c]{-90}{$=$} & & \rotatebox[origin=c]{-90}{$\geq$} \\
\rho(G) \ar@/_1.5pc/[rrrrrr]_{1:1} & < & \mathrm{SO}(2k+1) & \xrightarrow{1:1} & \mathrm{PSO}(2k+1) 
& > & \mathrm{Inn}(Q^r_{D_{2 k + 1}}) \\
}
\]
Here $\overline{W_{B_{2k+1}}}$ is the image of $W_{B_{2k+1}}$ by the quotient map.
Since the cardinalities of $\mathrm{Inn}(Q^r_{D_{2k+1}})$, 
$\mathrm{Inn}(Q_{D_{2k+1}})$ and $\overline{W_{B_{2k+1}}}$ coincide, 
we conclude that $\mathrm{Inn}(Q^r_{D_{2k+1}}) = \mathrm{Inn}(Q_{D_{2k+1}})$.

When $n = 2 k$ and $k > 1$, we have the following diagram:
\[
\xymatrix@C=1pt@R=1pt{
W_{D_{2k}} \ar@{>}[rrrrrr]^{2:1} & & & & & & \mathrm{Inn}(Q_{D_{2k}}) \\
\rotatebox[origin=c]{-90}{$\lneq$} & & & & & & \rotatebox[origin=c]{-90}{$\leq$} \\
W_{B_{2k}} \ar@/^1.5pc/[rrrrrr]^{2:1} & < & \mathrm{O}(2k) & \xrightarrow{2:1} & \mathrm{PO}(2k) 
& > &  \overline{W_{B_{2k}}} \\
\rotatebox[origin=c]{-90}{$\gneq$}  & & \cup & & \rotatebox[origin=c]{-90}{$=$} & & \rotatebox[origin=c]{-90}{$\geq$} \\
\rho(G) \ar@/_1.5pc/[rrrrrr]_{2:1} & < & \mathrm{SO}(2k) & \xrightarrow{2:1} & \mathrm{PSO}(2k) 
& > & \mathrm{Inn}(Q^r_{D_{2k}}) \\
}
\]
Recall that $W_{D_{2k}}$ and $\rho (G)$ are 
the kernels of the maps $\mathrm{sgn}$ and $\det$ respectively 
(\ref{eq:W_and_rho_G_as_kernels}).
These maps reduces to maps from $\mathrm{PO}(2k)$ and we have the following short 
exact sequences 
\begin{equation}
\begin{array}{ccccccccc}
1 & \to & \mathrm{Inn}(Q_{D_{2k}}) & \to & \overline{W_{B_{2k}}} & \xrightarrow{\mathrm{sgn}} & \{ \pm 1 \} & \to & 1 ,  \\
1 & \to & \mathrm{Inn}(Q^r_{D_{2k}}) & \to & \overline{W_{B_{2k}}} & \xrightarrow{\mathrm{det}} & \{ \pm 1 \} & \to & 1 .  \\
\end{array}
\end{equation}
The isomorphism $W_{B_{2k}} \cong (\mathbb{Z} / 2 \mathbb{Z})^{2k} \rtimes S_{2k}$ and 
its explicit description in $\mathrm{O}(2k)$ enable us to conclude that
\begin{equation}
\label{eq:description_of_Inn_Q_D_2k_and_Q^r_D_2k}
\begin{split}
&\overline{W_{B_{2k}}} \cong (\mathbb{Z} / 2 \mathbb{Z})^{2k-1} \rtimes S_{2k}, \\
&\mathrm{Inn}(Q_{D_{2k}}) \cong (\mathbb{Z} / 2 \mathbb{Z})^{2k-2} \rtimes S_{2k}, \\
&\mathrm{Inn}(Q_{D_{2k}}) \cap \mathrm{Inn}(Q^r_{D_{2k}}) 
\cong (\mathbb{Z} / 2 \mathbb{Z})^{2k-2} \rtimes A_{2k},
\end{split}
\end{equation}
where $A_n$ is the alternating group of degree $n$.
We will prove the following in the next subsection \S \ref{subsec:proof_of_no_S_n}.
\begin{proposition}
\label{prop:no_S_n}
If $k \geq 3$, 
$\mathrm{Inn} (Q^r_{D_{2k}}) \cong (W_{B_{2k}} \cap \mathrm{SO}(2k))  / \{ \pm 1 \}$ 
does not contain a subgroup isomorphic to $S_{2k}$.
\end{proposition}
However, 
$\mathrm{Inn} (Q_{D_{2k}}) \cong \overline{W_{D_{2k}}} 
\cong (\mathbb{Z}/2\mathbb{Z})^{2k-2} \rtimes S_{2k}$ contains 
a subgroup isomorphic to $S_{2k}$.
We conclude that $\mathrm{Inn}(Q_{D_{2k}}) \ncong  \mathrm{Inn}(Q^r_{D_{2k}})$ for 
$k \geq 3$.

In summary, we have the following.
\begin{theorem}
\label{thm:inner_automorphism_group_of_Q^r_D_n}
Let $Q_{D_n}$ be the Coxeter quandle of type $D_n$ and 
$Q^r_{D_n}$ be the rotational $D_n$ quandle.
If $n$ is odd, we have $\mathrm{Inn} (Q^r_{D_n}) \cong \mathrm{Inn} (Q_{D_n})
\cong (\mathbb{Z}/2 \mathbb{Z})^{n-1} \rtimes S_n$.

If $n$ is even and $n \geq 4$, $\mathrm{Inn} (Q^r_{D_n})$ and $\mathrm{Inn} (Q_{D_n})$ 
have the same order, and have isomorphic index $2$ subgroups.
But $\mathrm{Inn} (Q^r_{D_n}) \ncong \mathrm{Inn} (Q_{D_n})$ for $n \geq 6$.
\end{theorem}

When $n = 4$, a GAP calculation shows that 
$\mathrm{Inn}(Q_{D_{4}}) \cong  \mathrm{Inn}(Q^r_{D_{4}})$.
When $n = 2$, since $\dim V \wedge V = 1$, $SO(2)$ acts trivially on $V \wedge V$.
Thus we have $\mathrm{Inn} (Q^r_{D_2}) \cong \{1 \}$.

By \ref{eq:description_of_Inn_Q_D_2k_and_Q^r_D_2k}, 
we have two short exact sequences:
\[
\begin{array}{ccccccccc}
1 & \to & (\mathbb{Z} / 2 \mathbb{Z})^{2k-2} \rtimes A_{2k}
& \to & \mathrm{Inn}(Q_{D_{2k}}) & \xrightarrow{\mathrm{det}} & \{ \pm 1 \} & \to & 1 ,  \\
1 & \to & (\mathbb{Z} / 2 \mathbb{Z})^{2k-2} \rtimes A_{2k}
& \to & \mathrm{Inn}(Q^r_{D_{2k}}) & \xrightarrow{\mathrm{sgn}} & \{ \pm 1 \} & \to & 1 .  \\
\end{array}
\]
Thus $\mathrm{Inn}(Q_{D_{2k}})$ and $\mathrm{Inn}(Q^r_{D_{2k}})$ are written as 
semi-direct products of $(\mathbb{Z} / 2 \mathbb{Z})^{2k-2} \rtimes A_{2k}$ by 
$\mathbb{Z}/2\mathbb{Z}$ but not isomorphic for $k \geq 3$. 

\subsection{Proof of Proposition \ref{prop:no_S_n}}
\label{subsec:proof_of_no_S_n}
In this subsection, we will prove Proposition \ref{prop:no_S_n}.
First of all, we collect some formulas on matrices of the form
$M(\sigma, \{ \varepsilon_i \} ) = (\varepsilon_i \delta_{i, \sigma(j)})_{i, j}$.
As an example, when $\sigma = \begin{pmatrix} 1 & 2 & 3 \end{pmatrix}$ and 
$\varepsilon_1 = \varepsilon_2 = 1$, $\varepsilon_3 = -1$, we have
$M(\sigma, \{ \varepsilon_i \} ) 
= \begin{pmatrix} 
0 & 0 & 1 \\ 
1 & 0 & 0 \\ 
0 & -1 & 0 
\end{pmatrix}$.

\begin{lemma}
\begin{align}
\label{eq:inverse}
&M(\sigma, \{ \varepsilon_i \})^{-1} = (\varepsilon_i \delta_{i, \sigma(j)})_{i, j}^{-1}
= (\varepsilon_j \delta_{j, \sigma(i)})_{i, j}, \\
\label{eq:multiplication}
&M(\sigma, \{ \varepsilon_i \}) M(\tau, \{ \eta_i \})
= (\varepsilon_i \eta_{\sigma^{-1}(i)} \delta_{i, \sigma \tau (j)}), \\
\label{eq:conjugate}
&M(\tau, \{ \eta_i \})^{-1} M(\sigma, \{ \varepsilon_i \}) M(\tau, \{ \eta_i \})
= (\varepsilon_{\tau(i)} \eta_{\tau(i)} \eta_{\sigma^{-1}\tau(i)} \delta_{i, \tau^{-1} \sigma \tau(j)}).
\end{align}
\end{lemma}
\begin{proof}
Since $M(\sigma, \{ \varepsilon_i \} )$ is an orthogonal matrix, 
we obtain (\ref{eq:inverse}).
For (\ref{eq:multiplication}), 
$M(\sigma, \{ \varepsilon_i \}) M(\tau, \{ \eta_i \})
= (\varepsilon_i \delta_{i, \sigma(k)} \eta_{k} \delta_{k,\tau (j)})
= (\varepsilon_i \eta_{\sigma^{-1}(i)} \delta_{\sigma^{-1}(i),\tau (j)})
= (\varepsilon_i \eta_{\sigma^{-1}(i)} \delta_{i,\sigma\tau (j)})$.
For (\ref{eq:conjugate}), 
$M(\tau, \{ \eta_i \})^{-1} M(\sigma, \{ \varepsilon_i \}) M(\tau, \{ \eta_i \})
= (\eta_k \delta_{k,\tau(i)}  \varepsilon_k \eta_{\sigma^{-1}(k)} \delta_{k,\sigma\tau (j)})
= (\eta_{\tau(i)} \varepsilon_{\tau(i)} \eta_{\sigma^{-1} \tau(i)} \delta_{\tau(i),\sigma\tau (j)})
= (\eta_{\tau(i)} \varepsilon_{\tau(i)} \eta_{\sigma^{-1} \tau(i)} \delta_{i,\tau^{-1}\sigma\tau (j)})
$.
\end{proof}

\begin{lemma}
\label{lem:permute_signs}
Let $\sigma$ be a transposition $\begin{pmatrix} 1 & 2 \end{pmatrix}$.
Let $\tau$ be a permutation fixing $1$ and $2$.
Then 
\[
M(\tau, \{ \eta_i \})^{-1} M(\sigma, \{ \varepsilon_i \}) M(\tau, \{ \eta_i \})
= \begin{pmatrix} 
0 & \varepsilon_1 \eta_1 \eta_2 & & & O \\
\varepsilon_2 \eta_1 \eta_2 & 0 & & &  \\
& & \varepsilon_{\tau(3)}& & \\
& & & \ddots \\
O & & & &\varepsilon_{\tau(n)}
\end{pmatrix}.
\]
\end{lemma}
\begin{proof}
We have $M(\tau, \{ \eta_i \})^{-1} M(\sigma, \{ \varepsilon_i \}) M(\tau, \{ \eta_i \})
= (\varepsilon_{\tau(i)} \eta_{\tau(i)} \eta_{\sigma^{-1}\tau(i)} \delta_{i, \sigma(j)})$
by (\ref{eq:conjugate}).
By the assumption, $\tau(1) = 1$, $\sigma^{-1}\tau(1) = 2$, 
$\tau(2) = 2$, $\sigma^{-1}\tau(2) = 1$,
and $\sigma^{-1}\tau(i) = \tau(i)$  for $i = 3, \cdots , n$, 
we obtain the assertion. 
\end{proof}

\begin{proof}[Proof of Proposition \ref{prop:no_S_n}]
Assume that there exists an injective homomorphism 
$i : S_{2k} \hookrightarrow (W_{B_{2k}} \cap \mathrm{SO}(2k))  / \{ \pm 1 \}$ 
when $k \geq 3$.

The short exact sequence (\ref{eq:short_exact_sequence_for_rho(G)}) 
(see also (\ref{eq:rho_G})) reduces to the short exact sequence
\begin{equation}
\label{eq:short_exact_sequence_for_the_quotient}
0 \to (\mathbb{Z} / 2 \mathbb{Z})^{2k-2} \to 
(W_{B_{2k}} \cap \mathrm{SO}(2k))/\{ \pm 1 \} \xrightarrow[]{\pi} S_{2k} \to 1.
\end{equation}
Let $\pi : (W_{B_{2k}} \cap \mathrm{SO}(2k))  / \{ \pm 1 \} \to S_{2k}$ be the
surjective homomorphism in (\ref{eq:short_exact_sequence_for_the_quotient}).
Explicitly, we have $\pi( \pm M(\sigma, \{ \varepsilon_i \})) = \sigma$.
It is elementary to show that any normal subgroup of $S_n$ is isomorphic to 
$\{1\}$, $A_n$ or $S_n$ for $n \geq 5$ since $A_n$ is simple.
Thus $\mathrm{Ker} (\pi \circ i)$ is $\{1\}$, $A_{2k}$ or $S_{2k}$.
If $\mathrm{Ker} (\pi \circ i) = A_{2k}$, or $S_{2k}$, 
then $i(A_{2k})$ is included in $(\mathbb{Z} / 2 \mathbb{Z})^{2k-2}$.
This contradicts the fact that $A_{2k}$ is non-abelian.
Thus $\mathrm{Ker} (\pi \circ i) = \{ 1 \}$, 
that is $\pi |_{i(S_{2k})} : i(S_{2k}) \to S_{2k}$ is an isomorphism.

Let $\sigma = \begin{pmatrix} 1 & 2 \end{pmatrix} \in S_{2k}$.
Then $(\pi |_{i(S_{2k})})^{-1} ( \sigma ) \in i (S_{2k})$ is  written as 
$\pm M(\sigma, \{ \varepsilon_i \})$ 
for some $\varepsilon_i = \pm 1$ ($i = 1, \cdots , 2k$);
\[
\pm M(\sigma, \{ \varepsilon_i \}) 
= \pm
\begin{pmatrix} 
0 & \varepsilon_1 & & & O \\
\varepsilon_2 & 0 & & &  \\
& & \varepsilon_{3}& & \\
& & & \ddots \\
O & & & &\varepsilon_{2k}
\end{pmatrix}.
\]
Since $M(\sigma, \{ \varepsilon_i \})$ has order 2, 
we have $\varepsilon_1 = \varepsilon_2 = 1$ or $\varepsilon_1 = \varepsilon_2 = -1$. 
Since $\det M(\sigma, \{ \varepsilon_i \}) = 1$, 
there is an odd number of $\varepsilon_3, \cdots , \varepsilon_{2k}$ equal to $1$
($-1$ respectively).
For any permutation $\tau \in S_{2k}$ fixing $1$ and $2$, 
$(\pi |_{i(S_{2k})})^{-1} ( \tau ) \in i (S_{2k})$ is  written as 
$\pm M(\tau, \{ \eta_i \})$ 
for some $\eta_i = \pm 1$ ($i = 1, \cdots , 2k$).
By Lemma \ref{lem:permute_signs}, 
the conjugation of $M(\sigma, \{ \varepsilon_i \})$ by $M(\tau, \{ \eta_i \})$ 
permutes $\varepsilon_3, \cdots , \varepsilon_{2k}$.
Thus the inverse image $(\pi |_{i(S_{2k})})^{-1} ( \{ \sigma \} ) \subset  i (S_{2k})$
has more than 1 elements. 
This contradicts the fact that 
$\pi |_{i(S_{2k})} : i(S_{2k}) \to S_{2k}$ is an isomorphism.
\end{proof}

The proof does not work for $k = 2$.
In fact, the subgroup of $W_{B_4} \cap \mathrm{SO}(4)$ generated by 
$\begin{pmatrix} 
0 & -1 & 0 & 0 \\
1 & 0 & 0 & 0 \\
0 & 0 & 1 & 0 \\
0 & 0 & 0 & 1
\end{pmatrix}$
and 
$\begin{pmatrix} 
0 & 0 & -1 & 0 \\
0 & 1 & 0 & 0 \\
1 & 0 & 0 & 0 \\
0 & 0 & 0 & 1
\end{pmatrix}$
is isomorphic to $S_4$. 
But the image of this subgroup under $\pi : W_{B_4} \cap \mathrm{SO}(4) \to S_4$ 
is isomorphic to $S_3$. 
Thus the homomorphism $\pi |_{i(S_{2k})} : i(S_{2k}) \to S_{2k}$ in the proof 
may not be an isomorphism.

\section{Vendramin's classification}
\label{sec:vendramin's_classification}
Vendramin classified connected quandles of order $\leq 35$ in \cite{Vendramin}.
The list was expanded in \cite{RIG} for connected quandles of order $\leq 47$, 
Using Vendramin's GAP package RIG, 
we determined $Q_{W}$, $DQ_{W}$, $Q^r_{D_n}$ in the Vendramin's list.
The results are collected in Table \ref{table:small_quandles}.
In the table, $C_n$ is the cyclic group, $A_n$ is the alternating group, 
$S_n$ is the symmetric group, and $A : B$ is a semi-direct product $A \rtimes B$.
We remark that $|Q_{A_n}| = |DQ_{A_n}|/2 = n(n+1)/2$ 
and $|Q_{D_n}| = |DQ_{D_n}|/2 = n(n-1)/2$
since the order of $DQ_W$ coincides with the number of roots of $W$.
Among them, 
$Q_{6,1} \cong Q_{A_3} = Q_{D_3}$, $Q_{6,2} \cong Q^r_{D_3}$, 
$Q_{12,8} \cong Q_{D_4}$, $Q_{12,9} \cong Q^r_{D_4}$
are described as \emph{Galkin quandles} in \cite[Table 1]{CEHSY}, 
which were introduced by Clark, Elhamdadi, Hou, Saito and Yeatman.

We also collect the inner automorphism groups of $Q_W$, $DQ_W$ and $\Phi_W$
for exceptional finite Coxeter groups in Table \ref{table:exceptional}. 
We remark that $\mathrm{Inn} (\Phi_{W}) \cong W$ 
by Lemma \ref{lem:inner_automorphism_group_of_coxeter_rack}.

%
%

\begin{table}
\caption{$Q_W$, $DQ_W$, $Q^r_{D_n}$ in Vendramin's classification.}
\begin{tabular}{l|l|l|l}
order & name & description & $\mathrm{Inn}(Q)$ (GAP's structure description)\\
\hline
6 & $Q_{6,1}$ &  $Q_{A_3} = Q_{D_3}$ & $S_4$ \\ \cline{2-4}
  & $Q_{6,2}$ &  $Q^r_{D_3}$ & $S_4$ \\ \hline
10 & $Q_{10,1}$ & $Q_{A_4}$ & $S_5$ \\ \hline 
12 & $Q_{12,1}$ & $DQ_{A_3} = DQ_{D_3}$ & $S_4$ \\ \cline{2-4}
 & $Q_{12,8}$ & $Q_{D_4}$ & $((C_2 \times C_2 \times C_2 \times C_2) : C_3) : C_2$ \\ \cline{2-4}
 & $Q_{12,9}$ & $Q^r_{D_4}$ & $((C_2 \times C_2 \times C_2 \times C_2) : C_3) : C_2$ \\ \hline
15  & $Q_{15,2}$ & $Q_{H_3}$ & $S_6$ \\ \cline{2-4}
  & $Q_{15,7}$ & $Q_{A_5}$ & $S_6$ \\ \hline
20 & $Q_{20,3}$ & $DQ_{A_4}$ & $S_5$ \\ \cline{2-4}
  & $Q_{20,9}$ & $Q_{D_5}$  & $(C_2 \times C_2 \times C_2 \times C_2) : S_5$ \\ \cline{2-4}
  & $Q_{20,10}$ & $Q^r_{D_5}$ & $(C_2 \times C_2 \times C_2 \times C_2) : S_5$ \\ \hline
21 & $Q_{21,9}$ & $Q_{A_6}$ & $S_7$ \\ \hline
24 & $Q_{24,17}$ & $DQ_{D_4}$ & $(((C_2 \times C_2 \times C_2) : (C_2 \times C_2)) : C_3) : C_2$ \\ \hline
28 & $Q_{28,13}$ & $Q_{A_7}$ & $S_8$ \\ \hline
30 & $Q_{30,1}$ & $DQ_{H_3}$ & $A_5$ \\ \cline{2-4}
  & $Q_{30,16}$ & $DQ_{A_5}$ & $S_6$\\ \cline{2-4}
  & $Q_{30,23}$ & $Q^r_{D_6}$ & $((C_2 \times C_2 \times C_2 \times C_2) : A_6) : C_2$ \\ \cline{2-4}
  & $Q_{30,24}$ & $Q_{D_6}$ & $(C_2 \times C_2 \times C_2 \times C_2) : S_6$\\ \hline
36 & $Q_{36,72}$ & $Q_{E_6}$ & $\mathrm{O}(5,3) : C_2$\\ \cline{2-4}
 & $Q_{36,73}$ & $Q_{A_8}$ & $S_9$ \\ \hline
40 & $Q_{40,12}$ & $DQ_{D_5}$ & $(C_2 \times C_2 \times C_2 \times C_2) : S_5$ \\ \hline
42 & $Q_{42,21}$ & $DQ_{A_6}$ & $S_7$ \\ \cline{2-4}
 & $Q_{42,22}$ & $Q_{D_7}$ & $(C_2 \times C_2 \times C_2 \times C_2 \times C_2 \times C_2) : S_7$ \\ \cline{2-4}
 & $Q_{42,23}$ & $Q^r_{D_7}$ & $(C_2 \times C_2 \times C_2 \times C_2 \times C_2 \times C_2) : S_7$ \\ \hline
45 & $Q_{45,45}$ & $Q_{A_9}$ & $S_{10}$ \\
\end{tabular}
\label{table:small_quandles}
\end{table}

\begin{table}
\caption{Inner automorphism groups of $Q_W$, $DQ_W$, $\Phi_W$ for exceptional finite Coxeter groups.}
\begin{tabular}{l|l|l|l}
description & order & $|\mathrm{Inn} (Q)|$ & $\mathrm{Inn}(Q)$ (GAP's structure description) \\
\Xhline{1pt}
$Q_{E_6}$ & 36 & 51840 & $\mathrm{O}(5,3) : C_2$ \\ \hline
$DQ_{E_6}$ & 72 & 51840 & $\mathrm{O}(5,3) : C_2$ \\ \hline
$\Phi_{E_6}$ & 72 & 51840 & $\mathrm{O}(5,3) : C_2$ \\ \Xhline{1pt}
$Q_{E_7}$ & 63 & 1451520 & $\mathrm{O}(7,2)$ \\ \hline
$DQ_{E_7}$ & 126 & 1451520 & $\mathrm{O}(7,2)$ \\ \hline
$\Phi_{E_7}$ & 126 & 2903040 & $C_2 \times \mathrm{O(7,2)}$ \\ \Xhline{1pt}
$Q_{E_8}$ & 120 & 348364800 & $\mathrm{O}^+(8,2) : C_2$ \\ \hline
$DQ_{E_8}$ & 240 & 696729600 & $(C_2 . \mathrm{O}^+(8,2)) : C_2$ \\ \hline
$\Phi_{E_8}$ & 240 & 696729600 & $(C_2 . \mathrm{O}^+(8,2)) : C_2$ \\ \Xhline{1pt}
$Q_{F_4}$ & 24 & 576 & $((A_4 \times A_4) : C_2) : C_2$ \\ \hline
$DQ_{F_4}$ & 48 & 1152 & $((((C_2 \times C_2 \times C_2) : (C_2 \times C_2)) : (C_3 \times C_3)) : C_2) : C_2$ \\ \hline
$\Phi_{F_4}$ & 48 & 1152 & $((((C_2 \times C_2 \times C_2) : (C_2 \times C_2)) : (C_3 \times C_3)) : C_2) : C_2$ \\ \Xhline{1pt}
$Q_{H_3}$ & 15 & 60 & $A_5$ \\ \hline
$DQ_{H_3}$ & 30 & 60 & $A_5$ \\ \hline
$\Phi_{H_3}$ & 30 & 120 & $C_2 \times A_5$ \\ \Xhline{1pt}
$Q_{H_4}$ & 60 & 7200 & $(A_5 \times A_5) : C_2$ \\ \hline
$DQ_{H_4}$ & 120 & 14400 & $(\mathrm{SL}(2,5) : A_5) : C_2$ \\ \hline
$\Phi_{H_4}$ & 120 & 14400 & $(\mathrm{SL}(2,5) : A_5) : C_2$ \\ 
\end{tabular}
\label{table:exceptional}
\end{table}

\end{document}